\documentclass[anon,12pt]{colt2023} 

\title[Improved LMC for stochastic optimization via landscape modification]{Improved Langevin Monte Carlo for stochastic optimization via landscape modification}
\usepackage{times}
\graphicspath{{Fig/}}
\usepackage{amssymb}
\usepackage{amsmath}
\makeatletter

\newcommand{\Rmnum}[1]{\expandafter\@slowromancap\romannumeral #1@}
\makeatother

\numberwithin{equation}{section}
\newtheorem{assumption}{Assumption}[section]

\usepackage[backend=biber,style=apa,citestyle=authoryear,uniquename=false,uniquelist=false]{biblatex}
\usepackage{amsmath}
\usepackage{hyperref}

\usepackage{graphicx}

\usepackage{epstopdf}

\usepackage{float}
\restylefloat{table}

\usepackage{comment}

\coltauthor{%
 \Name{Michael C.H. Choi} \Email{mchchoi@nus.edu.sg}\\
 \addr Department of Statistics and Data Science and Yale-NUS College, National University of Singapore, Singapore
 \AND
 \Name{Youjia Wang} \Email{191840250@smail.nju.edu.cn}\\
 \addr Department of Mathematics, Nanjing University, China%
}

\begin{document}
\title{Improved Langevin Monte Carlo for stochastic optimization via landscape modification}
\maketitle

\begin{abstract}%
  Given a target function $H$ to minimize or a target Gibbs distribution $\pi_{\beta}^0 \propto e^{-\beta H}$ to sample from in the low temperature, in this paper we propose and analyze Langevin Monte Carlo (LMC) algorithms that run on an alternative landscape as specified by $H^f_{\beta,c,1}$ and target a modified Gibbs distribution $\pi^f_{\beta,c,1} \propto e^{-\beta H^f_{\beta,c,1}}$, where the landscape of $H^f_{\beta,c,1}$ is a transformed version of that of $H$ which depends on the parameters $f,\beta$ and $c$. While the original Log-Sobolev constant affiliated with $\pi^0_{\beta}$ exhibits exponential dependence on both $\beta$ and the energy barrier $M$ in the low temperature regime, with appropriate tuning of these parameters and subject to assumptions on $H$, we prove that the energy barrier of the transformed landscape is reduced which consequently leads to polynomial dependence on both $\beta$ and $M$ in the modified Log-Sobolev constant associated with $\pi^f_{\beta,c,1}$. This yield improved total variation mixing time bounds and improved convergence toward a global minimum of $H$. We stress that the technique developed in this paper is not only limited to LMC and is broadly applicable to other gradient-based optimization or sampling algorithms.
\end{abstract}


\acks{The first author acknowledges the financial support of the startup grant from the National University of Singapore and Yale-NUS College. He also acknowledges the support from MOE Tier 1 Grant under the Data for Science and Science for Data collaborative scheme.}


\section{Introduction}

In this paper, we are primarily interested in designing improved Langevin Monte Carlo algorithms for sampling in the low temperature regime or stochastic optimization. These tasks arise naturally in a wide range of domains broadly related to machine learning including but not limited to Bayesian computation, computational physics and chemistry as well as theoretical computer science. Precisely, given a possibly non-convex target function $H:\mathbb{R}^d\rightarrow \mathbb{R}$ to minimize, we are interested in (approximately) sampling from the associated Gibbs distribution $\pi^0_{\beta} \propto e^{-\beta H}$, where $\beta$ is the inverse temperature parameter. At low enough temperature and subject to appropriate assumptions on $H$, as the Gibbs distribution $\pi^0_{\beta}$ concentrates on the set of global minima of $H$ owing to the Laplace method \parencite{hwang1980laplace}, sampling from $\pi^0_{\beta}$ thus serves the purpose of finding an approximate global minimizer of $H$.

Among various stochastic algorithms that have been developed to sample from $\pi^0_{\beta}$, there is a recent resurgence of interests on algorithms that are based upon the Langevin Monte Carlo (LMC) algorithm. Intuitively speaking, LMC can be understood as a version of gradient descent perturbed by Gaussian noise. It arises as the Euler-Maruyama discretization of the overdamped Langevin diffusion, which follows the iteration as
\begin{equation}
    \quad X_{k+1}=X_{k}-\eta \nabla H(X_k)+\sqrt{2\eta/\beta} \cdot Z_k,\label{langevin monte carlo}
\end{equation}
where $k$ is the iteration step, $\eta$ is the step size, and $\{Z_k\}_{k\geq 0}$ are i.i.d. standard Gaussian distributions. Various theoretical tools have been utilized to analyze the dynamics of LMC and its many variants, for instance non-asymptotic analysis and functional inequalities \parencite{raginsky2017non,chewi2021analysis, cheng2018underdamped}, hitting time analysis and Cheeger's constant \parencite{zhang2017hitting, zou2021faster, chen2020stationary}, to name but a few.

In the setting where $H$ is convex (or equivalently $\pi_{\beta}^0$ is log-concave), the analysis of LMC or its variants is relatively well-understood, see for example the papers \parencite{brosse2018promises, dalalyan2019user} and the references therein. On the other hand, when $H$ is non-convex, it is known that the underlying overdamped Langevin diffusion converges slowly towards stationarity. One of the main bottlenecks hindering convergence stem from the presence of the so-called energy barrier $M$ separating a local minimum and a global minimum, and it is shown in \parencite{gayrard2005metastability} that the spectral gap is of the order of $\Theta(e^{\beta M})$ at low temperature. In a broad sense, $M$ can be understood as a measure of the difficulty of the landscape $H$. Note that $M$ also appears in the design of cooling schedule in the context of simulated annealing, see for example \parencite{holley1988simulated, chiang1987diffusion,monmarche2018hypocoercivity}. This highlights the importance of designing accelerated or improved algorithms in such challenging yet important situation.

In the literature, the design of accelerated Langevin-based stochastic algorithms can be broadly categorized into two types, namely to propose new dynamics on the same landscape or to alter the underlying landscape or geometry. For the first type, various variants of the original LMC dynamics have been investigated such as mirror Langevin \parencite{zhang2020wasserstein}, simulated tempering Langevin \parencite{lee2018beyond}, or variance-reduced Langevin \parencite{kinoshita2022improved}. The path taken in this paper is in the spirit of the second type in which we propose a new modification of the landscape. Similar ideas that utilize landscape modification as an acceleration technique have been exploited in other related areas under possibly a different name, such as the energy transformation algorithm \parencite{catoni1998energy} by Catoni and quantum adiabatic annealing \parencite{wang2016quantum}. The primary motivation of the modification is taken from \parencite{fang1997improved}, in which an overdamped Langevin diffusion with state-dependent diffusion coefficient is investigated.

\subsection{Our contributions}\label{subsec:our contributions}
 Given a possibly non-convex target function $H$, the core idea of landscape modification rests on transforming $H$ into another function that enjoys superior mixing or optimization performance than on the original landscape as specified by $H$. More precisely, in this paper we propose a new transformation $H_{\beta,c,1}^f$ of $H$, which is defined to be
\begin{equation}
    {H}_{\beta, c, 1}^f(x):=H^*+\int_{H^*}^{H(x)} h_{\beta,c,1}^f(u) du,\label{landscape modification function}
\end{equation}
where $H^* := \min_x H(x)$ is the global minimum value of $H$, and $h_{\beta,c,1}^f:\mathbb{R}\rightarrow \mathbb{R}^+$ is a positive function that takes on the form of
\begin{equation}
    h_{\beta,c,1}^f(u):=\frac{1}{\beta f(u-c)+1}\label{h(x)}.
\end{equation}
Note that this transformation $H_{\beta,c,1}^f$ depends on three parameters, namely:
\begin{itemize}
    \item $\beta$ is the inverse temperature parameter as in the dynamics \eqref{langevin monte carlo}.

    \item $f:\mathbb{R}\rightarrow \mathbb{R}^+$ is a non-negative and non-decreasing function that satisfies $f\in C^1(\mathbb{R})$ and $f(0) = f^{\prime}(0) = 0$. Note that these conditions imply that $f(x) = 0$ for $x \leq 0$. In this paper, we focus ourselves on the choice of $f$ as specified by \eqref{eq:our f} below. 

    \item $c$ is known as the threshold parameter and is assumed to be chosen such that $c>H^*$. In the sublevel set $\{H(x)\leq c\}$, we note from \eqref{landscape modification function} that $H^f_{\beta,c,1} = H$, and hence the transformed and original landscape coincides. On the other hand, in the superlevel set $\{H(x) > c\}$, we observe that the modified landscape $H^f_{\beta,c,1}$ is compressed when $\beta$ is large enough, with its value is specified by $f$. As such, the parameters $f$ and $c$ jointly determine and control the transformation above the threshold $c$.
    
    \end{itemize}
    
    When $H$ is differentiable, owing to $h > 0$, it can readily be seen that the set of critical points of $H_{\beta,c,1}^f$ coincides with that of $H$. In the sequel, we shall be considering the modified Gibbs distribution $\pi_{\beta,c,1}^f\propto \exp\left(-\beta H_{\beta,c,1}^f\right)$. In the special case when $f = 0$, we see that $H_{\beta,c,1}^0 = H$, and as such the original Gibbs distribution $\pi^0_{\beta} = \pi_{\beta,c,1}^0$. This explains the superscript of $0$ in the original Gibbs distribution.
    
    As an illustration of the landscape modification, we plot a benchmark function $H(x)=\sin 2x+\frac{5}{2}\cos x+x^2-1.38$ and its associated $H_{\beta,c,1}^f(x)$ in Figure \ref{fig:landscape}, where $\delta$ is a hyperparameter in the choice of $f$, which will be explained in Section \ref{main results}.

\begin{figure}
    \centering
    \begin{minipage}[t]{0.7\textwidth}
    \includegraphics[width=10cm]{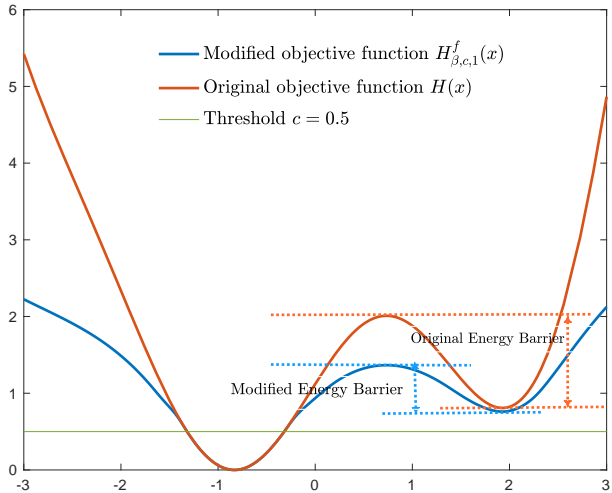}
    \caption{Landscape of original and modified objective function. Here $H(x)=\sin 2x+\frac{5}{2}\cos x+x^2-1.38$,
        $c=0.5, \beta=1, \delta=1$, and $f$ is specified in \eqref{eq:our f} below.}
    \label{fig:landscape}
    \end{minipage}
\end{figure}

With the above landscape modification idea in mind, we shall consider the Langevin Monte Carlo algorithm running instead on the modified landscape $H_{\beta,c,1}^f$, that is, we consider 
\begin{equation*}
    Y_{k+1}=Y_k-\eta \nabla H_{\beta,c,1}^f(Y_k)+\sqrt{2\eta/\beta} \cdot Z_k,
\end{equation*}
where $Y = \{Y_k\}_{k\geq 0}$ is the iteration of the modified dynamics, and the gradient of the modified objective function can be computed as
\begin{equation*}
    \nabla H_{\beta,c,1}^f=\frac{\nabla H}{\beta f(H-c)+1}.
\end{equation*}
Note that we need the information of both the gradient $\nabla H$ and the objective $H$ at each step, while the original Langevin Monte Carlo algorithm only requires the gradient $\nabla H$ in each iteration. From this perspective additional computational cost is incurred per iteration in $Y$. We are now in a position to state the key contributions of this paper:
\begin{itemize}
    \item We propose and analyze new Langevin Monte Carlo algorithm on a modified landscape suitable for stochastic optimization or sampling in the low temperature with minimal additional computational cost.

    \item The proof utilizes results from the metastability literature that effectively bounds various functional constants such as spectral gap and the Log-Sobolev constant by means of the associated energy barrier, thus bypassing the need of dissipativeness assumption as is typical in the literature \parencite{raginsky2017non}. For example, we shall prove that the modified objective function $H_{\beta,c,1}^f$ enjoys a reduced energy barrier $M^f$ than the original energy barrier $M$, which consequently leads to the property that the modified Gibbs distribution $\pi_{\beta,c,1}^f\propto \exp\left(-\beta H_{\beta,c,1}^f\right)$ has a Log-Sobolev constant that depends at most polynomially on $\beta$ and $M$ with $C_{LSI}^f=\mathcal{O}\left(\beta M+\frac{\beta M^2}{H(m_2)-c}\right)$, which is in sharp contrast with the Log-Sobolev constant of the original Gibbs distribution $\pi_{\beta}^0\propto \exp\left(-\beta H\right)$ that  exponentially depends on $\beta$ and $M$ as in \parencite{menz2014poincare}.
\end{itemize}

\section{Preliminaries}\label{preliminaries}
Before we present our main results, in this section we shall recall a few classical results and fix some notations that will be used throughout the paper. First, we formally define the stochastic dynamics of Langevin Monte Carlo.
\begin{definition}\label{LMC}
    (Langevin Monte Carlo (LMC), $X = \{X_k\}$) Suppose $\eta>0$ is the step size, then LMC can be written as 
\begin{equation}
    X_{k+1}=X_k-\eta \nabla H(X_k)+\sqrt{2\eta/\beta}\cdot Z_k,\label{standard LMC}
\end{equation}
where $Z_k\sim N(0,I)$ is an independent standard Gaussian distribution, and $X_0\sim\rho_0$ where $\rho_0$ can be arbitrary initial distribution. Let $\rho_k$ denote the probability distribution of $X_k$ that evolves following LMC.
\end{definition}

Under certain regularity conditions on $H$, it can be shown that the iteration \eqref{standard LMC} will converge to $\pi_\eta$, the stationary distribution of \eqref{standard LMC}. As $\eta\rightarrow{0}$, the LMC recovers the overdamped Langevin diffusion in continuous time. Moreover, one should note that $\pi_{\eta}\neq \pi_{\beta}^0$, and there exists a small stepsize-dependent bias between $\pi_{\eta}$ and $\pi_{\beta}^0$ which is in proportion to $\eta$. The explicit bias between $\pi_{\eta}$ and $\pi_{\beta}^0$ can be found in Section 7 in \parencite{vempala2019rapid}. 

Next, we recall the definitions of Poincar\'e, Log-Sobolev and Talagrand Inequalities. 
\begin{definition}
    (Poincar\'e and Log-Sobolev Inequalities, PI and LSI) A Borel probability measure $\nu$ on $\mathbb{R}^d$ satisfies the Poincar\'e inequality with constant $C_{PI}>0$, if for all test functions $\phi \in H^1(\nu)$,
$$
 \int\left(\phi-\mathbb{E}_\nu[\phi] \right)^2 d \nu \leq C_{PI} \int|\nabla \phi|^2 d\nu,
$$
where $H^1(\nu)=W^{2,1}(\nu)$ is the Sobolev space with respect to measure $\nu$. In a similar vein, the probability measure $\nu$ satisfies the Log-Sobolev inequality with constant $C_{LSI}>0$, if for all test function $\phi: \mathbb{R}^d \rightarrow \mathbb{R}^{+}$ with $I(\phi \nu \mid \nu)<\infty$ holds
$$
 \int \phi \log \frac{\phi}{\mathbb{E}_\nu[\phi]} d\nu \leq \frac{C_{LSI}}{2} \int \frac{|\nabla \phi|^2}{2 \phi} d \nu \left(=:I(\phi\nu|\nu)\right)
$$
\end{definition}

\begin{assumption}\label{assumption 6}
    (Talagrand Inequality)
 A Borel probability measure $\nu$ on $\mathbb{R}^d$ satisfies Talagrand Inequality with a constant $T_a > 0$, $i.e.$ for arbitrary probability measure $\rho$, we have
$$
\frac{T_a}{2} W_2(\rho, \nu)^2 \leq D_{KL}(\rho\|\nu),
$$
where $D_{KL}(\rho\|\nu)$ is the Kullback-Leiberg (KL) divergence of $\rho$ with respect to $\nu$, that is,
$$
D_{KL}(\rho\|\nu)= \mathbb{E}_{\rho}\left[\ln \frac{\rho}{\nu}\right]= \int_{\mathbb{R}^d} \rho(x) \ln \frac{\rho(x)}{\nu(x)} d x,
$$
where (with a slight abuse of notations) $\rho(\cdot)$ and $\nu(\cdot)$ are the respective density of $\rho$ and $\nu$.
\end{assumption}

Thirdly, we introduce the class of objective functions $H$ that we will be working with in this paper. The following definitions and assumptions are drawn from \parencite{menz2014poincare}, which give growth conditions on $H$ that ensure the finiteness of the set of stationary points of $H$.
\begin{definition}
    (Morse function) A smooth function $H: \mathbb{R}^d \rightarrow \mathbb{R}$ is a Morse function, if the Hessian $\nabla^2 H$ is nondegenerated on the set of critical points, $i.e.$ for some $1 \leq C<\infty$ holds
 $$
 \quad \forall x \in \mathcal{S}:=\left\{x \in \mathbb{R}^d: \nabla H=0\right\}: \frac{|\xi|^2}{C} \leq\left\langle \xi, \nabla^2 H(x) \xi\right\rangle \leq C|\xi|^2 .
 $$
\end{definition}

\begin{assumption}\label{assumption 1}
    (Growth condition in PI) $H \in C^3\left(\mathbb{R}^d, \mathbb{R}\right)$ is a Morse function, such that for some constants $C_H>0$ and $D_H \geq 0$ holds
$$
\begin{array}{rr}
\left(\mathrm{A} 1_{\mathrm{PI}}\right) & \liminf _{|x| \rightarrow \infty}|\nabla H| \geq C_H, \\
\left(\mathrm{A} 2_{\mathrm{PI}}\right) & \liminf _{|x| \rightarrow \infty}\left(|\nabla H|^2-\Delta H\right) \geq-D_H.
\end{array}
$$
\end{assumption}

\begin{assumption}
(Growth condition in LSI)\label{assumption 2}
$H \in C^3\left(\mathbb{R}^d, \mathbb{R}\right)$ is a Morse function, such that for some constants $C>0$ and $D \geq 0$ holds
$$
\begin{array}{ll}
\left(\mathrm{A} 1_{\mathrm{LSI}}\right) & \liminf _{|x| \rightarrow \infty} \frac{|\nabla H(x)|^2-\Delta H(x)}{|x|^2} \geq C, \\
\left(\mathrm{A} 2_{\mathrm{LSI}}\right) & \inf _x \nabla^2 H(x) \geq-D .
\end{array}
$$
\end{assumption}
The assumptions above ensure that the set of stationary points is discrete and finite. In particular, we define $\mathcal{S} :=\{m_1,m_2,...,m_S\}$ as the set of local minimum, where $S<\infty$.

Fourth, we recall the notion of saddle height and energy barrier, which measures the difficulty of the landscape of $H$ in a broad sense.

\begin{definition}
    (Saddle height)  The saddle height $\widehat{H}\left(m_i, m_j\right)$ between two local minimum $m_i, m_j$ is defined by
\begin{equation*}
\widehat{H}\left(m_i, m_j\right):=\inf \left\{\max _{s \in[0,1]} H(\gamma(s)): \gamma \in C\left([0,1], \mathbb{R}^n\right), \gamma(0)=m_i, \gamma(1)=m_j\right\}.
\end{equation*}
\end{definition}

\begin{assumption}\label{assumption 3}
    (Non-degeneracy, unique global minimum and energy barrier) There exists $\delta>0$ such that
    \begin{enumerate}
        \item The saddle height between two local minimum $m_i, m_j$ is attained at a unique critical point $s_{i, j} \in \mathcal{S}$ of index one, $i.e.$ $H\left(s_{i, j}\right)=$ $\widehat{H}\left(m_i, m_j\right)$, and if $\left\{\lambda_1, \ldots, \lambda_n\right\}$ denote the eigenvalues of $\nabla^2 H\left(s_{i, j}\right)$, then it holds $\lambda_1<0$ and $\lambda_i>0$ for $i=2, \ldots, n$. 

        \item The set of local minimum $\mathcal{S}=\left\{m_1, \ldots, m_S\right\}$ is ordered such that $m_1$ is a global minimum and for all $i \in\{3, \ldots, S\}$ yields
$$
H\left(s_{1,2}\right)-H\left(m_2\right) \geq H\left(s_{1, i}\right)-H\left(m_i\right)+\delta .
$$
        \item $m_1$ is the unique global minimum of $H$.
        
        \item  The energy barrier $M$ of $H$ is defined by
    \begin{equation*}
        M=H(s_{1,2})-H(m_2).
    \end{equation*}
    \end{enumerate}
\end{assumption}

While the assumption of unique global minimum of $H$ seems to be restrictive, this setting does indeed arise in practice for instance in the investigation of the Gaussian mixture model with an unique global mode, see \parencite{chen2020likelihood}. We shall also impose the following $L$-smooth and Lipschitz Hessian assumption on $H$.

\begin{assumption}\label{assumption 4}
    ($L$-smooth)
H is twice differentiable and $\|\nabla H(x)-\nabla H(y)\|\leq L\|x-y\|$.
\end{assumption}
\begin{assumption}\label{assumption 5}
    (Lipschitz Hessian)
$\nabla^2 H$ is $L_1$-Lipschitz, $i.e.$ $\|\nabla^2 H(x)-\nabla^2 H(y)\|\leq L_1\|x-y\|$. 
\end{assumption}

We now recall a few classical results arising from the metastability literature that will be frequently used throughout the paper. The following proposition comes from Corollary 2.18 in \parencite{menz2014poincare}, which bounds various functional constants in terms of the energy barrier.

\begin{proposition}\label{PI and LSI}
    Under Assumption \eqref{assumption 1} to \eqref{assumption 3}, let 
     $\left\{\kappa_i^2\right\}_{i=1}^S$ be given by 
$$
\kappa_i^2:=\operatorname{det} \nabla^2 H\left(m_i\right) .
$$
On one hand, if one has one unique global minimum, namely $H\left(m_1\right)<$ $H\left(m_i\right)$ for $i \in\{2, \ldots, S\}$, it holds
$$
C_{PI} \lesssim \frac{1}{\kappa_2} \frac{2 \pi  \sqrt{\left|\operatorname{det} \nabla^2\left(H\left(s_{1,2}\right)\right)\right|}}{\beta\left|\lambda^{-}\left(s_{1,2}\right)\right|} \exp \left(\beta \left(H(s_{1,2})-H(m_2)\right)\right),
 $$
 $$
C_{LSI} \lesssim\left(\beta\left(H(m_2)-H(m_1)\right)+\ln \left(\frac{\kappa_1}{\kappa_2}\right)\right) C_{PI} .
 $$
On the other hand, if $H\left(m_1\right)=H\left(m_2\right)<H\left(m_i\right)$ for $i \in\{3, \ldots, M\}$, it holds
$$
C_{PI} \lesssim \frac{1}{\kappa_1+\kappa_2} \frac{2 \pi  \sqrt{\left|\operatorname{det} \nabla^2\left(H\left(s_{1,2}\right)\right)\right|}}{\beta\left|\lambda^{-}\left(s_{1,2}\right)\right|} \exp \left(\beta \left(H(s_{1,2})-H(m_2)\right)\right),
$$
$$
\quad C_{LSI} \lesssim \frac{1}{\Lambda\left(\kappa_1, \kappa_2\right)} \frac{2 \pi  \sqrt{\left|\operatorname{det} \nabla^2\left(H\left(s_{1,2}\right)\right)\right|}}{\beta\left|\lambda^{-}\left(s_{1,2}\right)\right|} \exp \left(\beta \left(H(s_{1,2})-H(m_2)\right)\right),
$$
where $\Lambda(p,q):=\frac{p-q}{\ln p-\ln q}$, and $\lambda^-(s_{1,2})$ is the negative eigenvalue of $\nabla^2 H(s_{1,2})$. The symbol $\lesssim$ means $\leq$ up to a multiplicative error $1+\mathcal{O}\left(|\ln \beta|^{3/2}\beta^{-1/2}\right)$ as explained in Theorem 2.12 in \parencite{menz2014poincare}.
\end{proposition}

It can readily be seen that the upper bounds of $C_{PI}$ and $C_{LSI}$ in both cases are composed of the term $e^{\beta M}$ multiplied by a prefactor which is a rational polynomial comprising of the parameters related to $H$. We shall focus on the dominant term $e^{\beta M}$ where $M$ is the energy barrier, that is, $M=H(s_{1,2})-H(m_2)$. Our proposed landscape modification thus seeks to mitigate the energy barrier by transformation to hopefully yield smaller upper bounds on PI and LSI constants, while at the same time we seek to ensure the prefactor in front of the exponential term does not depend exponentially on both $\beta$ and $M$.

The following proposition is taken from Theorem 1 in \parencite{vempala2019rapid}, which gives an upper bound on the KL divergence between the $k^{th}$ iterate of LMC to the Gibbs distribution $\pi_{\beta}^0$.

\begin{proposition} \label{mixing KL}
Under Assumption \ref{assumption 4}, for the Langevin Monte Carlo defined in Definition \ref{LMC}, suppose $\nu=\pi_{\beta}^0$ is the $L$-smooth Gibbs distribution satisfying LSI with constant $C_{LSI}>0$. For any $X_0 \sim \rho_0$ with $D_{KL}(\rho_0\|\nu)<\infty$, the iterates $X_k \sim \rho_k$ of LMC with step size $0<\eta \leq \frac{\beta C_{LSI}}{4 L^2}$ satisfy
$$
D_{KL}(\rho_k\|\nu)\leq e^{-\frac{\eta k}{\beta C_{LSI}}} D_{KL}(\rho_0\|\nu)+8 \eta d L^2\cdot C_{LSI}/\beta,
$$
where we recall that $D_{KL}(\rho\|\nu)$ is the KL-divergence of $\rho$ with respect to $\nu$. 
\end{proposition}
We can see that a smaller $C_{LSI}$ can indeed lead to a smaller upper bound in Proposition \ref{mixing KL}.

\section{Main results}\label{main results}

In this section, we state the main results of the paper. The core idea of landscape modification rests on executing the LMC algorithm on an alternative landscape $H_{\beta,c,1}^f$ with a reduced energy barrier. We shall compare $Y = \{Y_k\}$ that runs on the alternative landscape of $H_{\beta, c, 1}^f$ as an alternative to the original LMC $X = \{X_k\}$, that is,
\begin{align}
(\text{Original LMC dynamics})&\quad X_{k+1}=X_k-\eta \nabla H(X_k)+\sqrt{2\eta/\beta}\cdot Z_k,\\
(\text{Modified LMC dynamics})&\quad Y_{k+1}=Y_k-\eta \nabla H_{\beta,c,1}^f(Y_k)+\sqrt{2\eta/\beta}\cdot Z_k,\label{modified LMC}
\end{align}
where $Z_k\sim N(0,I)$, and $X_0\sim \rho_0, Y_0\sim \rho_0^f$ are the respective initial distributions for $X$ and $Y$.

Now, we turn our discussion to the choice of the parameter $f$ in landscape modification. First of all, we recall that we restrict ourselves to consider $f \in C^1(\mathbb{R})$, since this condition ensures that the Hessian of $H_{\beta,c,1}^f$ is continuous (see for instance the proof of Lemma \ref{LSI Constant} below), which is required and is essential to the proofs of our main results. As $f_1(x) := x\vee 0$ unfortunately is not differentiable at $x=0$, in order to guarantee that $f\in C^1(\mathbb{R})$, we shall add an oscillation to $f_1(x)=x\vee 0$ in a small neighborhood $(0,\delta)$ on the right side of $x=0$ to define the following $f$:
\begin{align}\label{eq:our f}
    f(x)=f(x,\delta) :=
\begin{cases}
    0, &\mathrm{ for }\quad x\leq 0,\\
    \frac{20}{3\delta} x^2, &\mathrm{ for }\quad 0<x\leq\frac{\delta}{4},\\
    -\frac{20}{3\delta} (x-\frac{\delta}{2})^2+\frac{5}{6}\delta, &\mathrm{ for }\quad \frac{\delta}{4}<x\leq \frac{\delta}{2},\\
    A\exp (-\frac{1}{B(x-\frac{\delta}{2})^2})+\frac{5}{6}\delta, &\mathrm{ for }\quad \frac{\delta}{2}<x\leq\delta,\\
    x, &\mathrm{ for }\quad x> \delta,
\end{cases}
\end{align}
where $\delta$ is a parameter to be chosen, and $A=\frac{\delta}{6} e^{3/2}$, $B=\frac{8}{3\delta^2}$, therefore $f(x)\in C^1(\mathbb{R})$. In the rest of this paper, this choice of $f$ is kept fixed. In the following sections, we will show that a small enough $\delta$ can lead to the same effect on reducing the energy barrier as that of $f_1$. This is as intuitively expected since $f_1$ and $f$ are asymptotically equivalent for large $x$. To illustrate both $f_1$ and $f$, these are plotted in Figure \ref{fig:f(x)} with the choice of $\delta = 1$. To demonstrate the effect of landscape modification with this choice of $f$, we recall that in Section \ref{subsec:our contributions} we compare both the modified objective function $H_{\beta,c,1}^f$ as well as the original objective function $H$ in Figure \ref{fig:landscape} on a benchmark objective function.

\begin{figure}
    \centering
    \begin{minipage}[t]{0.7\textwidth}
     \centering
     \includegraphics[width=9cm]{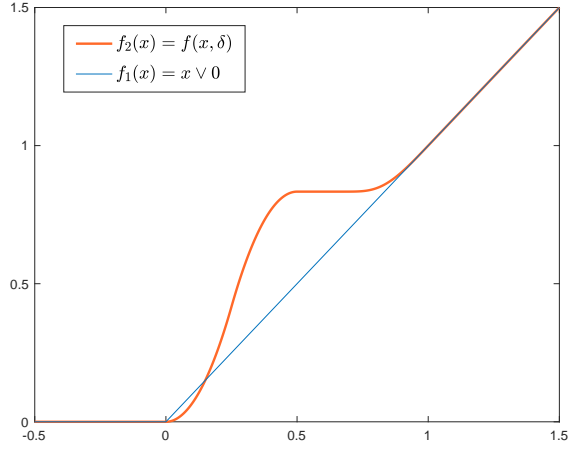}
     \caption{Oscillation to $f_1(x)=x\vee 0$ within $(0,\delta)$, where $f_2(x)=f(x,\delta)$ and $\delta=1$.}
     \label{fig:f(x)}
    \end{minipage}

\end{figure}

From the perspective of sampling, we are interested in the mixing performance of the modified LMC $Y$ towards the original Gibbs distribution $\pi_{\beta}^0\propto e^{-\beta H}$, while from the perspective of optimization, we shall investigate the stochastic optimization of $H$ via the modified LMC $Y$.

We shall impose an additional assumption which assures that the modified distribution $\pi_{\beta,c,1}^f$ exists, or equivalently, under which it is guaranteed that $\exp\left(-\beta H_{\beta,c,1}^f\right)$ is integrable in $\mathbb{R}^d$. This result is to be proved in Section \ref{existence of modified distribution}.

\begin{assumption}\label{assumption integrable}
    Suppose the temperature is low enough with $\beta>\frac{1}{\delta}$ and $c$ is chosen such that $H^*<c<H(m_2)-\delta$. In addition, there exists $I_1=I_1(c,\delta) < \infty$ such that \begin{gather}
        \int_{H>c+\delta}\frac{2\delta}{H-c}dx<I_1.\label{integrable}
    \end{gather}
\end{assumption}

This assumption can be readily verified for a broad class of Morse functions, for instance the class of Morse functions with quadratic growth, as is common in the simulated annealing literature \parencite{monmarche2018hypocoercivity}.

\subsection{Some commonly used notations}
We now formally introduce various notations that will be used in various parts of the paper. In general, the variables with a superscript of $f$ are associated with the modified LMC $Y$, while the objects without any superscript or with a superscript of $0$ are associated with the original LMC $X$. For reader's convenience, a table of commonly used notations is summarized in Table \ref{tab:notation} in Appendix \ref{appendix A}. In the list below, we highlight several important notations:

\begin{itemize}
    \item Denote $G(x)=H_{\beta,c,1}^f(x)$ as the modified objective function. Under Assumption \ref{assumption 2} and \ref{assumption 3}, we can find that $x^*=m_1$ is the global minimum, and $m_2$ is the second smallest local minimum. We also denote $H^*=H(x^*)$ to be the global minimum of $H$ at $x^*$.

    \item $\rho_k$ and $\rho_k^f$ are the distributions of the original iteration $X_k$ and modified iteration $Y_k$ respectively, that is, $X_k \sim \rho_k, Y_k\sim \rho_k^f$. 
We shall denote $\pi_1=\pi_{\beta}^0\propto e^{-\beta H}$ to be the original Gibbs distribution, while $\pi_2=\pi_{\beta,c,1}^f\propto \exp\left(-\beta H_{\beta,c,1}^f\right)$ as the modified Gibbs distribution, which $\pi_1$ and $\pi_2$ are two shorthand notations. We also write $Z_1=\int e^{-\beta H}dx$ and $Z_2=\int \exp\left(-\beta H_{\beta,c,1}^f\right)dx$ to be the respective normalization constants for $\pi_1$ and $\pi_2$. 

    \item $C_{LSI}^f$ is the Log-Sobolev constant of $\pi_{\beta,c,1}^f$, and $C_{PI}^f$ is the Poincar\'e constant of $\pi_{\beta,c,1}^f$. Also, $G=H_{\beta,c,1}^f$ is $L^f$-smooth (to be proved in Theorem \ref{theorem 1}). The constants of Talagrand Inequality for $\pi_1$ and $\pi_2$ are respectively $T_a$ and $T_a^f$, and $M$ and $M^f$ are the energy barriers of $H$ and $G$ respectively. The set of local minima of $G$ is defined to be the set $\mathcal{S}^f=\{m_1^f,m_2^f,\ldots,m_S^f\}$, and $s_{1,2}^f$ is the saddle point between $m_1^f$ and $m_2^f$.

    \item Denote $K=\nabla ^2 H(m_1)$, and $\lambda_{min}(K)=\gamma>0$ to be the smallest eigenvalue of $K$ since $K$ is a positive definite matrix by Assumption \ref{assumption 3}. For arbitrary positive definite matrix $A\in \mathbb{R}^{d\times d}$ and $x\in \mathbb{R}^d$, the matrix norm $\|x\|_A^2=\frac{1}{2}x^TAx$, and $\|B\|$ denotes the operator norm of matrix $B$. We also write $\omega_d$ to be the volume of the unit ball in $d$-dimensional Euclidean space, and $\mu$ to be the Lebesgue measure in $\mathbb{R}^d$. 

    \item For any two probability measures $\rho$ and $\nu$ on $\mathbb{R}^d$, we denote $D_{TV}(\rho,\nu)=\frac{1}{2}\int |\rho-\nu|dx$ as the total variation distance between $\rho$ and $\nu$. Denote $W_2(\rho,\nu)$ as the Wasserstein distance between $\rho$ and $\nu$.

\end{itemize}

\subsection{Improved bound on total variation mixing for sampling from $\pi_{\beta}^0$}

In the following three subsections, we present the main results of the paper. We shall see that the time complexity of the modified LMC depends crucially on the modified LSI $C^f_{LSI}$, which has a polynomial dependence on both $\beta$ and $M$ under our assumptions (i.e. with appropriate parameters tuning) while the original LMC exhibits an exponential dependence on both $\beta$ and $M$.

Our first main result gives non-asymptotic error bounds on utilizing the modified LMC $Y$ to approximately sample from the original Gibbs distribution $\pi_{\beta}^0$ at low-temperature:

\begin{theorem}\label{theorem 1}
    Under Assumption \ref{assumption 1} to \ref{assumption 5} and \ref{assumption integrable},
    then for any $Y_0\sim \rho_0^f$ and step size $0<\eta\leq \frac{\beta C_{LSI}^f}{16L^{f^2}}$, denote $\beta_0=\max\left\{\frac{1}{\delta^2}, \frac{3}{20\delta},\frac{1}{\delta}\right\}$, and when $\beta>\beta_0$, 
    we have 
\begin{align*}
 D_{TV}(\rho_k^f,\pi_{\beta}^0) &\leq e^{-\frac{\eta k}{2\beta C_{LSI}^f}}\sqrt{2D_{KL}(\rho_0^f\|\pi_{\beta,c,1}^f)}+4\sqrt{\eta d L^{f^2}\cdot C_{LSI}^f/\beta}\\
 &\quad + \left(\frac{\beta L}{2\pi}\right)^{d/2}\left(\mu(c\leq H \leq c+\delta)e^{-\beta (c - H^*)}+e^{-\beta (c - H^*)}I_1 \right),
\end{align*}
where $C_{LSI}^f=\mathcal{O}\left(\beta M+\frac{\beta M^2}{H(m_2)-c} \right)$, and $G$ is $L^f$-smooth with 
$$L^f\leq \max\left\{(2(c-H^*)+3\delta)L, \left(\frac{80}{3}(c-H^*)+\frac{92}{3}\delta\right) L, \frac{c-H^*}{2}+\frac{\delta}{8}+\frac{\delta}{2}L\right\}.$$
\end{theorem}
The proof of Theorem \ref{theorem 1} is deferred to Section \ref{sampling performance}. We observe that, for the first term to be smaller than or equal to a given error tolerance $\epsilon$, the number of iterations $k$ required is at most 
\begin{equation*}
k = \mathcal{O}\left(\dfrac{\beta C^f_{LSI}}{\eta} \log\left(\dfrac{1}{\epsilon}\right)\right),
\end{equation*}
which depends on the inverse temperature $\beta$ and energy barrier $M$ at most polynomially. Moreover, in order for the second term to be smaller than $\epsilon$, the step size $\eta$ can be picked with $\eta=\Theta\left(\epsilon^2\right)$, which can be independent of $\beta$. 

\subsection{Improved convergence towards the global minimum}\label{main: convergence rate to global minimum}

In our second main result, we present an improved convergence rate of iterate $Y_k$ towards the global minimum $x^*$ by giving an upper bound on the probability of reaching the superlevel set of the form 
$$\mathbb{P}\left(H(Y_k)>H^*+\epsilon\right),$$
where $\epsilon$ is a given precision or error tolerance.

\begin{theorem}\label{comparison of probability}
    Under Assumption \ref{assumption 1} to \ref{assumption 5} and \ref{assumption integrable}, for any given precision with $0<\epsilon< c - H^*$, step size $0<\eta\leq \min\left\{\frac{\beta C_{LSI}^f}{16L^{f^2}},\frac{\beta C_{LSI}}{16L^2}\right\}$ and low enough temperature with $\beta>\beta_0=\max\left\{\frac{1}{\delta^2}, \frac{3}{20\delta},\frac{1}{\delta}\right\}$, if there exists $I_2 < \infty$ such that
    \begin{equation*}
        \int_{\mathbb{R}^d}e^{-(H-H^*)}dx<I_2,
    \end{equation*}
    then we have 
    \begin{align*}
        \mathbb{P}(H(X_n)>H^*+\epsilon) &\leq 2e^{-\frac{\eta n}{2\beta C_{LSI}}}\sqrt{2D_{KL}(\rho_0\|\pi_{\beta}^0)}+8\sqrt{\eta d L^2\cdot C_{LSI}/\beta}\\
        &\quad + \left(\frac{\beta L}{2\pi}\right)^{d/2}e^{-(\beta-1)\epsilon}I_2,
    \end{align*}
    while
    \begin{align*}
        \mathbb{P}(H(Y_k)>H^*+\epsilon)
        &\leq 2e^{-\frac{\eta k}{2\beta C_{LSI}^f}}\sqrt{2D_{KL}(\rho_0^f\|\pi_{\beta,c,1}^f)}+8\sqrt{\eta d L^{f^2}\cdot C_{LSI}^f/\beta}\\
        &\quad + \left(\frac{\beta L}{2\pi}\right)^{d/2}\left(2\mu(c\leq H\leq c+\delta)e^{-\beta (c-H^*)}+2e^{-\beta (c-H^*)}I_1+e^{-(\beta-1)\epsilon}I_2\right),
    \end{align*}
where we recall that $C_{LSI}^f$ and $L^f$ are as stated in Theorem \ref{theorem 1}.
\end{theorem}
The proof of Theorem \ref{comparison of probability} is presented in Section \ref{convergence rate to global minimum}. We see that the third term on the right hand side of Theorem \ref{comparison of probability} above converges to zero as $\beta \to \infty$, and hence at low enough temperature this term places a negligible role. For the first term to be smaller than or equal to a given error tolerance $\epsilon$, the number of iterations $k$ and $n$ required are at most 
\begin{align*}
k &= \mathcal{O}\left(\dfrac{\beta C^f_{LSI}}{\eta} \log\left(\dfrac{1}{\epsilon}\right)\right), \quad 
n = \mathcal{O}\left(\dfrac{\beta C_{LSI}}{\eta} \log\left(\dfrac{1}{\epsilon}\right)\right).
\end{align*}
As a result, this upper bound on $n$ exhibits an exponential dependence on both $\beta$ and $M$ while the analogous upper bound on $k$ has at most a polynomial dependence on both $\beta$ and $M$.

\subsection{Reduced excess risk to the global minimum}

Our final main result concerns bounding the difference $\mathbb{E}_{Y_k\sim\rho_k^f}[H(Y_k)]-H^*,$
that we call the excess risk to the global minimum.

\begin{theorem}\label{result of expectation under H}
    Under Assumption \ref{assumption 1} to \ref{assumption 5}, \ref{assumption 6} and \ref{assumption integrable}, for given $0<r<\frac{3\gamma}{2L_1}$, suppose that $x^*$ is the unique local minimum of $H$ within $B_r(x^*)$ and $\max\limits_{B_r(x^*)}H(x)-H^*=\delta_r\leq c-H^*$. Assume that $\min\limits_{B_r' (x^*)}H(x)-H^*=\Delta_r>0$, then for step size $0<\eta\leq \frac{\beta C_{LSI}^f}{16L^{f^2}}$,
    and low enough temperature with $\beta>\beta_0=\max \left\{\frac{1}{\delta},\frac{1}{\delta^2}, \frac{3}{20\delta}\right\}$, if there exists $0 < I_4=I_4(c) < \infty,$ such that
\begin{equation*}
    \int_{H>c}(H-H^*)e^{-(H-H^*)}dx<I_4,
\end{equation*}
then we have 
\begin{align*}
    \mathbb{E}_{Y_k\sim\rho_k^f}[H(Y_k)]-H^*
    &\leq 4LC_{LSI}^f\left(e^{-\frac{\eta k}{\beta C_{LSI}^f}} D_{KL}(\rho_0^f\|\pi_{\beta,c,1}^f)+8\eta d L^{f2}\cdot C_{LSI}^f/\beta\right)\\
    &\quad +4LC_{LSI}^f\left(\frac{\beta L}{2\pi}\right)^{d/2}\left(\mu(c<H\leq c+\delta)e^{-\beta (c-H^*)}+e^{-\beta (c-H^*)}I_1\right)\\
    &\quad +3\frac{d\omega_d \Gamma(\frac{d}{2}+1)}{\sqrt{\det K}}\left(\frac{2L}{\pi}\right)^{d/2}\cdot\frac{1}{\beta}\\
    &\quad +2\left(\frac{\beta L}{2\pi}\right)^{d/2}\left((c-H^*)\mu(H\leq c)e^{-\beta \Delta_r}+e^{-(\beta-1)(c-H^*)}I_4\right),
\end{align*}
where we recall that $C_{LSI}^f$ and $L^f$ are as stated in Theorem \ref{theorem 1}.
\end{theorem}
The proof of Theorem \ref{result of expectation under H} and its intuition can be found in Section \ref{excess risk to global minimum}. Next, we compare Theorem \ref{result of expectation under H} with classical result in Remark \ref{rk:classical} which concerns the excess risk without landscape modification. We note that the central difference in excess risks originates from the different Log-Sobolev constants: specifically, $C_{LSI}=\mathcal{O}\left(e^{\beta M}\right)$, while $C_{LSI}^f=\mathcal{O}\left(\beta M+\frac{\beta M^2}{H(m_2)-c}\right)$. Despite a tradeoff of adding the distance between $\pi_{\beta}^0$ and $\pi_{\beta,c,1}^f$ into the computation of the excess risk of modified LMC, these play a relatively less significant contributions than that of the Log-Sobolev constant in computing the time complexity.

\begin{remark}\label{rk:classical}
   Under Assumption \ref{assumption 1} to \ref{assumption 5}, \ref{assumption 6} and \ref{assumption integrable}, and the same setting in Theorem \ref{result of expectation under H}, for step size $0<\eta\leq \frac{\beta C_{LSI}^f}{16L^{f^2}}$,
   if 
    \begin{equation*}
        \int_{H>c} (H-H^*)e^{-(H-H^*)}dx<I_4, 
    \end{equation*}
   then by Proposition \ref{mixing KL}, we have 
\begin{equation*}
       W_2^2\left(\rho_k,\pi_{\beta}^0\right) \leq 2C_{LSI}\left(e^{-\frac{\eta k}{\beta C_{LSI}}} H_\nu\left(\rho_0\right)+8 \eta d L^2\cdot C_{LSI}/\beta \right),
\end{equation*}
and
\begin{align*}
    \mathbb{E}_{X\sim\pi_1}&[H(X)]-H^*\\
    &\leq \frac{3}{2}\frac{d\omega_d \Gamma(\frac{d}{2}+1)}{\sqrt{\det K}}\left(\frac{2L}{\pi}\right)^{d/2}\cdot\frac{1}{\beta}
    +\left(\frac{\beta L}{2\pi}\right)^{d/2}\left((c-H^*)\mu(H\leq c)e^{-\beta \Delta_r}+e^{-(\beta-1)(c-H^*)}I_4\right),
\end{align*}
which implies 
\begin{align*}
    \mathbb{E}_{X_k\sim\rho_k}[H(X_k)]-H^*
    &\leq  LW_2^2\left(\rho_k,\pi_{\beta}^0\right)+2(\mathbb{E}_{X\sim\pi_1}[H(X)]-H^*)\\
    &\leq 2LC_{LSI}\left(e^{-\frac{\eta k}{\beta C_{LSI}}} H_\nu\left(\rho_0\right)+8 \eta d L^2\cdot C_{LSI}/\beta \right)\\
    &\quad + 3\frac{d\omega_d \Gamma(\frac{d}{2}+1)}{\sqrt{\det K}}\left(\frac{2L}{\pi}\right)^{d/2}\cdot\frac{1}{\beta}\\
    &\quad +2\left(\frac{\beta L}{2\pi}\right)^{d/2}\left(c\mu(H\leq c)e^{-\beta \Delta_r}+e^{-(\beta-1)(c-H^*)}I_4\right),
\end{align*}
where $C_{LSI}=\mathcal{O}\left(e^{\beta M}\right)$, and the proof of the first inequality can be found in Lemma \ref{expectation into W2}.
\end{remark}

\printbibliography

\appendix

\section{Proofs of the main results}\label{appendix B}

\subsection{Existence of modified Gibbs distribution $\pi_{\beta,c,1}^f$}\label{existence of modified distribution}

The aim of this subsection is to prove the existence of the modified Gibbs distribution $\pi_{\beta,c,1}^f$. In the following theorem, we will prove that under Assumption \ref{assumption integrable}, the modified Gibbs distribution exists. 

\begin{theorem}
    Under Assumption \ref{assumption 1} to \ref{assumption 5} and \ref{assumption integrable}, 
    the modified Gibbs distribution $\pi_{\beta,c,1}^f$ exists, i.e.
    \begin{equation*}
        \int_{\mathbb{R}^d}e^{-\beta G}dx<\infty.
    \end{equation*}
\end{theorem}
\begin{proof} First, we write down explicitly the relationship between $H$ and $G$. In $\{x:H(x)\leq c\}$, we have $H=G$. On the other hand in $\{x:H(x)>c+\delta\}$, we have
\begin{align*}
G(x)&=H^*+\int_{H^*}^{H(x)} \frac{1}{\beta f(u-c)+1}du\\
&=c+\int_c^{c+\delta}\frac{1}{\beta f(u-c)+1}du+\int_{c+\delta}^{H(x)}\frac{1}{\beta (u-c)+1}du\\
&=c+\int_c^{c+\delta}\frac{1}{\beta f(u-c)+1}du+\frac{1}{\beta}\ln \frac{H-c+1/\beta}{\delta+1/\beta}.
\end{align*}
Since $\beta>\frac{1}{\delta}$, we thus have
\begin{equation}
    c+\frac{1}{\beta}\ln \frac{H-c}{2\delta}\leq G(x)\leq 
    c+\delta+\frac{1}{\beta}\ln \frac{H-c}{\delta},\label{H and G}
\end{equation}
which leads to
\begin{equation*}
    \int_{H>c+\delta}e^{-\beta G}dx\leq e^{-\beta c}\int_{H>c+\delta}\frac{2\delta}{H-c}dx<e^{-\beta c}I_1.
\end{equation*}
Note that $\mu\left(H\leq c+\delta\right)<\infty$, the desired result follows.
\end{proof}

In many general cases, Assumption \ref{assumption integrable} can be readily verified. For example, in the case of quadratic growth of $H$, if there exists $R>0$ and $b_1,b_2>0$ such that when $\|x\|>R$, $b_1\|x\|^2<H(x)<b_2\|x\|^2$, then the condition \eqref{integrable} is satisfied.

\subsection{Proof for improved bound on total variation mixing for sampling from $\pi_{\beta}^0$}\label{sampling performance}
In this section, we shall give the proof of Theorem \ref{theorem 1}
via bounding the total variation distance between $\rho_k^f$ and $\pi_{\beta}^0$. 

Firstly, we compute the energy barrier $M^f$ of the modified target $G$, in terms of the original energy barrier $M$, the inverse temperature $\beta$ and the threshold parameter $c$: 
\begin{lemma}\label{energy barrier}
    (Reduced energy barrier of $G$)

    Under Assumption \ref{assumption 1} to \ref{assumption 5} and \ref{assumption integrable}, the energy barrier of modified function is given by
    \begin{equation*}
    M^f=G(s_{1,2}^f)-G(m_2^f) \leq\frac{1}{\beta}\cdot \ln\frac{M}{H(m_2)-c}.
    \end{equation*}
\end{lemma}
\begin{proof}
    First, let us recall that the modified function $G$ is defined by
    \begin{align*}
     G(x)=H_{\beta,c,1}^f(x)=H^*+\int_{H^*}^{H(x)}\frac{1}{\beta f (u-c)+1}du.
    \end{align*}
    Since $f(x)\geq 0$, we see that $G$ preserves the set of stationary points of $H$, and for $\forall x,y$, $H(x)>H(y)$ if and only if $G(x)>G(y)$.
    This implies that $s_{1,2}^f=s_{1,2}$ and $m_2^f=m_2$. Moreover, with our choice of $c$ that $H^*<c<H(m_2)-\delta$, we have 
    \begin{align*}
        M^f&=G(s_{1,2}^f)-G(m_2^f)=G(s_{1,2})-G(m_2)\\
        &=\int_{H(m_2)}^{H(s_{1,2})} \frac{1}{\beta (u-c)+1}du\\
        &=\frac{1}{\beta}\left(\ln\left(H(s_{1,2})-c+1/\beta\right)-\ln\left(H(m_2)-c+1/\beta\right)\right)\label{M'}\\ 
        &= \frac{1}{\beta}\cdot \ln\frac{H(s_{1,2})-H(m_2)}{H(m_2)-c+1/\beta}\\
        &\leq \frac{1}{\beta}\cdot \ln\frac{M}{H(m_2)-c}.
    \end{align*}
\end{proof}

Next, we shall plug the modified energy barrier $M^f$ into Proposition \ref{PI and LSI} to calculate the modified PI and LSI constants, namely $C_{PI}^f$ and $C_{LSI}^f$ respectively. 

\begin{lemma}\label{LSI Constant}
(Smaller upper bound on the Log-Sobolev constant)
Under Assumption \ref{assumption 1} to \ref{assumption 5} and \ref{assumption integrable}, we have 
$$C_{LSI}^f=\mathcal{O}\left(\beta M+\frac{\beta M^2}{H(m_2)-c} \right).$$
\end{lemma}
\begin{proof}
Note that when $H(x)>c+\delta$, the gradient and the Hessian of $G$ can be calculated as
    \begin{align*}
        \nabla G(x)&=\frac{\nabla H(x)}{1+\beta(H(x)-c)},\\
        \nabla^2 G(x)&=\frac{\nabla^2 H(x)}{\beta(H-c)+1}-
        \frac{\nabla H\cdot \nabla H^T\cdot \beta}{(\beta(H-c)+1)^2},
    \end{align*}
which lead to 
$$
\kappa_i^{f^2}:=\det \nabla ^2 G(m_i^f)=\frac{\kappa_i^2}{\left(\beta(H(m_i)-c)+1\right)^d}.
$$
Moreover, noting that $G(x)\leq H(x)$ and recalling $G(m_1)=H^*$, we thus have 
\begin{equation*}
G(m_2)-G(m_1)\leq H(m_2)-H^*.
\end{equation*}
Collecting the above and using both Proposition \ref{PI and LSI} and Lemma \ref{LSI Constant} give
\begin{align*}
    C_{PI}^f &\lesssim \frac{1}{\kappa_2} \frac{2 \pi  \sqrt{\left|\operatorname{det} \nabla^2\left(H\left(s_{1,2}\right)\right)\right|}}{\left|\lambda^{-}\left(s_{1,2}\right)\right|}\cdot 
    \frac{\beta(H(s_{1,2})-c)+1}{\beta}\cdot
    \frac{M}{H(m_2)-c}\\
    &= \frac{1}{\kappa_2} \frac{2 \pi  \sqrt{\left|\operatorname{det} \nabla^2\left(H\left(s_{1,2}\right)\right)\right|}}{\left|\lambda^{-}\left(s_{1,2}\right)\right|}\cdot 
    \left(M+\frac{M^2}{H(m_2)-c}+\frac{1}{\beta}\frac{M}{H(m_2-c)}
    \right)
\end{align*}
and
\begin{equation}
    C_{LSI}^f\lesssim \left(\beta (H(m_2)-H^*)+\ln \frac{\kappa_1}{\kappa_2}\right)\cdot C_{PI}^f,
\end{equation}
which implies $C_{LSI}^f=\mathcal{O}\left(\beta M+\frac{\beta M^2}{H(m_2)-c} \right)$.
\end{proof}

\begin{remark}\label{rk:compareLSI}
    We can see that $C_{LSI}^f$ is smaller compared with $C_{LSI}$, the LSI constant of the original Gibbs distribution $\pi_{\beta}^0$, which can be directly obtained from Proposition \ref{PI and LSI}. Precisely, for large $\beta$ we can write 
\begin{equation*}
    C_{LSI}=\mathcal{O}\left(e^{\beta M}\right),
\end{equation*}
while
\begin{equation*}
    C_{LSI}^f=\mathcal{O}\left(\beta M+\frac{\beta M^2}{H(m_2)-c} \right).
\end{equation*}
In other words, $C_{LSI}$ depends exponentially on both $\beta$ and $M$, while $C_{LSI}^f$, the modified constant, depends at most polynomially on these parameters, which is an significant improvement by contrast.
\end{remark}

By Proposition \ref{mixing KL} and Remark \ref{rk:compareLSI} above, we can note that the KL divergence of iteration $Y_k$ to $\pi_{\beta,c,1}^f$ is smaller owing to a reduced LSI constant. Next we estimate the asymptotic behavior of $\pi_{\beta,c,1}^f$ to $\pi_{\beta}^0$ as $\beta\rightarrow \infty$.

\begin{lemma}\label{total variation}
   Under Assumption \ref{assumption 1} to \ref{assumption 5} and \ref{assumption integrable}, we have
\begin{equation*}
    D_{TV}(\pi_{\beta}^0,\pi_{\beta,c,1}^f)\leq \left(\frac{\beta L}{2\pi}\right)^{d/2}\left(\mu(c\leq H\leq c+\delta)e^{-\beta (c - H^*)}+e^{-\beta (c - H^*)}I_1 \right),
\end{equation*}
where we recall $I_1$ appears as the upper bound in Assumption \ref{assumption integrable}, and $\mu$ is the Lebesgue measure on $\mathbb{R}^d$.
\end{lemma}
\begin{proof}
Recalling that $\mu_1(x)=e^{-\beta H}$, $\mu_2(x)=e^{-\beta G}$ and $H\geq G$, we have
\begin{align*}
    D_{TV}(\pi_{\beta}^0,\pi_{\beta,c,1}^f) &= \frac{1}{2}\int_{\mathbb{R}^d}\left |\frac{\mu_1}{Z_1}-\frac{\mu_2}{Z_2}\right |dx\\
    &\leq \frac{1}{2Z_1}\left(\int _{\mathbb{R}^d}|\mu_1-\mu_2|dx+\int_{\mathbb{R}^d}\left |\frac{Z_2-Z_1}{Z_2}\right |\mu_2 dx\right)\\
    &= \frac{1}{2Z_1}\int _{\mathbb{R}^d}|\mu_1-\mu_2|dx+\frac{1}{2Z_1}|Z_2-Z_1|\\
    &\leq \frac{1}{Z_1}\int_{\mathbb{R}^d} |\mu_1-\mu_2|dx\\
    &= \frac{1}{Z_1}\int_{\mathbb{R}^d} \left | e^{-\beta H}-e^{-\beta G}  \right |dx\\
    &\leq \frac{1}{Z_1} \int_{H\geq c}e^{-\beta G}dx\\
    &= \frac{1}{Z_1}\left(\int_{c\leq H\leq c+\delta}e^{-\beta G}dx+\int_{H>c+\delta}e^{-\beta G}dx \right)\\
    &\leq  \frac{1}{Z_1}\left(\mu(c\leq H\leq c+\delta)e^{-\beta c}+e^{-\beta c}I_1 \right),
\end{align*}
where we utilize the triangle inequality in the first inequality, and the last inequality follows from $G\geq c$ when $c\leq H\leq c+\delta$, and 
$G(x)\geq c+\frac{1}{\beta}\ln \frac{H-c}{2\delta}$ by equation \eqref{H and G} when $H(x)>c+\delta$. 
Moreover, note that $H(x)\leq H^*+\frac{L}{2}\|x-x^*\|^2$, which implies
\begin{equation}
    Z_1=\int_{\mathbb{R}^d} e^{-\beta H}dx\geq e^{-\beta H^*} \int_{\mathbb{R}^d} \exp\left(-\frac{\beta L}{2}\|x-x^*\|^2\right)dx=e^{-\beta H^*} \left(\frac{2\pi}{\beta L}\right)^{d/2},\label{lower bound of Z1}
\end{equation}
with which we obtain the result.
\end{proof}

Another technical lemma that we need, before showing the result about the total variation  between $\rho_k^f$ and $\pi_{\beta}^0$, stems from Property C.1 in \parencite{kinoshita2022improved} and illustrates the relationship between $H$ and $\nabla H$. 
\begin{lemma}\label{H and gradient H}
Under Assumption \ref{assumption 4} that $H$ is $L$-$smooth$, we have 
\begin{equation}
    H(x)-H^*\geq \frac{1}{2L}\|\nabla H(x)\|^2.
\end{equation}
\begin{proof}
    Denote $u(x)=H(x)-H^*$, then $u$ is also $L$-$smooth$.
Note that for a $L$-$smooth$ function $H$, there exists
$$
\forall x, y \in \mathbb{R}^d, H(y) \leq H(x)+\langle\nabla H(x), y-x\rangle+\frac{L}{2}\|y-x\|^2,
$$
then we have 
\begin{equation*}
    u\left(x-\frac{1}{L} \nabla u(x)\right) \leq u(x)-\frac{1}{L}\|\nabla u(x)\|^2+\frac{1}{2 L}\|\nabla u(x)\|^2=u(x)-\frac{1}{2 L}\|\nabla (x)\|^2.
\end{equation*}
By $u\geq 0$, we have 
\begin{equation*}
    u(x)=H(x)-H^*\geq \frac{1}{2L}\|\nabla H(x)\|^2.
\end{equation*}
\end{proof}
\end{lemma}
 
 Collecting the above Proposition \ref{mixing KL}, Lemma \ref{energy barrier}, \ref{LSI Constant}, \ref{total variation} and \ref{H and gradient H} yield the proof of one of the main results of this paper:

\begin{proof}[Proof of Theorem \ref{theorem 1}]
    First, using triangle inequality we see that
    \begin{align}\label{eq:DTV triangle}
        D_{TV}(\rho_k^f,\pi_{\beta}^0)\leq D_{TV}(\rho_k^f,\pi_{\beta,c,1}^f)+D_{TV}(\pi_{\beta,c,1}^f,\pi_{\beta}^0).
    \end{align}
    We apply directly Lemma \ref{total variation} to bound the second term in the right hand side of \eqref{eq:DTV triangle} to yield the stated result. As for the first term of \eqref{eq:DTV triangle}, by recalling the results from Proposition \ref{mixing KL}, and together with the Pinsker's inequality
    \begin{equation}
    D_{TV}(\nu_1,\nu_2)\leq \sqrt{2D_{KL}(\nu_1\|\nu_2)},\label{TV and KL}
    \end{equation}
    these give rise to
    \begin{align}
    D_{TV}(\rho_k^f,\pi_{\beta,c,1}^f)&\leq \sqrt{2e^{-\frac{\eta k}{\beta C_{LSI}^f}} D_{KL}(\rho_0^f\|\pi_{\beta,c,1}^f)+16 \eta d L^{f^2}\cdot C_{LSI}^f/\beta}\nonumber\\
    &\leq \sqrt{2e^{-\frac{\eta k}{\beta C_{LSI}^f}} D_{KL}(\rho_0^f\|\pi_{\beta,c,1}^f)}+4\sqrt{\eta d L^{f^2}\cdot C_{LSI}^f/\beta}.\label{L'}
    \end{align}
    Finally, we proceed to bound $L^f = \sup\|\nabla ^2 G(x)\|$. When $H(x)>c$, the Hessian of $G$ can then be written as 
    \begin{equation*}
        \nabla ^2 G(x)=\frac{\nabla ^2 H}{\beta f(H-c)+1}-
        \frac{\nabla H\cdot \nabla H^T\cdot \beta f'(H-c)}{(\beta f(H-c)+1)^2},
    \end{equation*}
then we see that
\begin{align*}
    \|\nabla ^2 G(x)\|&=\max\limits_{\|y\|=1}\left<y,\nabla ^2G(x)\cdot y\right>\\
    &\leq \frac{\|\nabla ^2H(x)\|}{\beta f(H-c)+1}+\frac{\beta f'(H-c)}{(\beta f(H-c)+1)^2}\cdot 
    \max\limits_{\|y\|=1}\left<y,\nabla H\cdot\nabla H^T y\right>\\
    &\leq\frac{L}{\beta f(H-c)+1}+\frac{\beta f'(H-c)}{(\beta f(H-c)+1)^2}\cdot \|\nabla H\|^2\\
    &=\frac{1}{\beta f(H-c)+1}\left(L+\frac{\beta f'(H-c)}{\beta f(H-c)+1}\|\nabla H\|^2\right).
\end{align*}
By Lemma \ref{H and gradient H}, when $H(x)>c+\delta$, $f'(x)=1$ and $\|\nabla H\|^2\leq 2L(H-H^*)$, since $\beta>\frac{1}{\delta^2}$, we have
\begin{align*}
    \|\nabla^2 G(x)\|&\leq 
    \frac{1}{\beta (H-c)}\left(L+\frac{\beta (H-H^*)}{\beta(H-c)+1}\cdot 2L\right)\\
    &\leq \frac{1}{\beta\delta}\left(3+\frac{2(c-H^*)}{\delta}\right)\cdot L\\
    &\leq (2(c-H^*)+3\delta)\cdot L,
\end{align*}
when $c+\frac{\delta}{4}\leq H(x)\leq c+\delta$, recalling that $f'(x)\leq f'(\frac{\delta}{4})=\frac{10}{3}$, for $\beta>\frac{1}{\delta^2}$, we have 
\begin{align*}
    \|\nabla^2 G(x)\|&\leq 
    \frac{1}{\beta (H-c)}\left(L+\frac{10}{3}\frac{\beta (H-H^*)}{\beta(H-c)+1}\cdot 2L\right)\\
    &\leq \frac{4}{\beta \delta}\left(L+\frac{20}{3}L\cdot\frac{c-H^*+\delta}{\delta}\right)\\
    &\leq \left(\frac{80}{3}(c-H^*)+\frac{92}{3}\delta\right)\cdot L,
\end{align*}
and when $c<H(x)<c+\frac{\delta}{4}$, $f(x)=\frac{20}{3\delta} x^2$ and $f'(x)=\frac{40}{3\delta}x$. For simplicity in the following calculation, denote $\alpha=\frac{20}{3\delta}\beta$, $H-c=t$ and $c-H^*+\frac{\delta}{4}=mL$, for $\beta>\frac{3}{20\delta}$, we can write
\begin{align*}
    \|\nabla^2 G(x)\|&\leq
    \frac{1}{\beta\frac{20}{3\delta}(H-c)^2+1}\left(L+\frac{\beta\frac{40}{3\delta}(H-c)}{\beta\frac{20}{3\delta}(H-c)^2+1}\left(c-H^*+\frac{\delta}{4}\right)\right)\\
    &=\frac{1}{\alpha t^2+1}\left(1+\frac{2\alpha t}{\alpha t^2+1}m\right)L\\
    &=\frac{\alpha t^2+2m\alpha t+1}{(\alpha t^2+1)^2}L
    =\frac{(\alpha t+\frac{1}{t})+2m\alpha}{(\alpha t+\frac{1}{t})^2}L\leq \frac{m\alpha+\sqrt{\alpha}}{2\alpha}L
    \\
    &=\frac{c-H^*}{2}+\frac{\delta}{8}+\frac{\delta}{2}L.
\end{align*}
Therefore, with the assumption that $H$ is $L$-smooth, when $\beta>\frac{1}{\delta^2}\vee \frac{3}{20\delta}$, we thus have $L^f\leq \max\left\{(2(c-H^*)+3\delta)L, \left(\frac{80}{3}(c-H^*)+\frac{92}{3}\delta\right) L, \frac{c-H^*}{2}+\frac{\delta}{8}+\frac{\delta}{2}L\right\}$. The desired result follows.
\end{proof}

\subsection{Proof for improved convergence towards the global minimum}\label{convergence rate to global minimum}

The focus of this section is to prove an upper bound on the large deviation probability of the form 
$$\mathbb{P}\left(H(Y_k)>H^*+\epsilon\right),$$
as stated in Theorem \ref{comparison of probability}. First, we relate these large deviation probability of $\{X_k\}$ and $\{Y_k\}$ with that of $\pi_1$ and $\pi_2$:

\begin{lemma}\label{upper bound of probability}
    Recalling $\pi_1=\pi_{\beta}^0\propto e^{-\beta H}$ and $\pi_2=\pi_{\beta,c,1}^f \propto e^{-\beta G}$. For any precision $\epsilon>0$, we have
    \begin{gather*}
        \left |\mathbb{P}\left(H(X_k)>H^*+\epsilon\right)-\int_{H>H^*+\epsilon}\pi_1dx
        \right|\leq 2D_{TV}(\rho_k,\pi_1),\\
        \left |\mathbb{P}\left(H(Y_k)>H^*+\epsilon\right)-\int_{H>H^*+\epsilon}\pi_2dx
        \right|\leq 2D_{TV}(\rho_k^f,\pi_2).
    \end{gather*}
Among which 
\begin{equation*}
   \left|\int_{H>H^*+\epsilon}\pi_1dx-\int_{H>H^*+\epsilon}\pi_2dx \right|\leq 2D_{TV}(\pi_1,\pi_2).
\end{equation*}
\end{lemma}
\begin{proof} First, we note that
    \begin{align*}
    \mathbb{P}(H(X_k)>H^*+\epsilon)
    &=\int_{H>H^*+\epsilon}\rho_kdx\\
    &=\int_{H>H^*+\epsilon} \rho_k dx-\int_{H>H^*+\epsilon}  \pi_1 dx+
    \int_{H>H^*+\epsilon} \pi_1 dx,
    \end{align*}
and
\begin{equation*}
    \left|\int_{H>H^*+\epsilon} \rho_k dx-\int_{H>H^*+\epsilon}  \pi_1 dx
    \right|\leq \int_{H>H^*+\epsilon}\left|\rho_k-\pi_1\right|dx
    \leq 2D_{TV}(\rho_k,\pi_1),
\end{equation*}
whence we have
\begin{equation*}
    \left |\mathbb{P}\left(H(X_k)>H^*+\epsilon\right)-\int_{H>H^*+\epsilon}\pi_1dx
        \right|\leq 2D_{TV}(\rho_k,\pi_1),
\end{equation*}
which completes the case of $X_k$.
The case of $Y_k$ is similar to that of $X_k$. In addition, we see that
\begin{equation*}
    \left|\int_{H>H^*+\epsilon} \pi_1 dx-\int_{H>H^*+\epsilon}  \pi_2 dx
    \right|\leq \int_{H>H^*+\epsilon}\left|\pi_1-\pi_2\right|dx
    \leq 2D_{TV}(\pi_1,\pi_2),
\end{equation*}
which completes the proof.
\end{proof}

Lemma \ref{upper bound of probability} gives an upper bound on the difference between the probability regarding the distribution of $k^{th}$ iteration and the associated Gibbs distribution. Moreover, by Lemma \ref{total variation}, we observe that $\int_{H>H^*+\epsilon}\pi_1dx$ and $\int_{H>H^*+\epsilon}\pi_2dx$ asymptotically converge to $0$ as $\beta\rightarrow \infty$. To obtain the quantitative convergence rate, we present the following lemma:

\begin{lemma}\label{upper bound of pi1 and pi2}
    Under Assumption \ref{assumption 1} to \ref{assumption 5} and \ref{assumption integrable}, for $\beta>\frac{1}{\delta}$ and $0<\epsilon<c$, if there exists $I_2 < \infty$ such that
    \begin{equation}
        \int_{\mathbb{R}^d}e^{-(H-H^*)}dx<I_2,\label{I2}
    \end{equation}
    then we have 
    \begin{align*}
        \int_{H>H^*+\epsilon}\pi_1 dx &\leq \left(\frac{\beta L}{2\pi}\right)^{d/2}e^{-(\beta-1)\epsilon}I_2,\\
        \int_{H>H^*+\epsilon}\pi_2 dx &\leq \left(\frac{\beta L}{2\pi}\right)^{d/2}\left(2\mu(c\leq H\leq c+\delta)e^{-\beta (c-H^*)}+2e^{-\beta (c-H^*)}I_1+e^{-(\beta-1)\epsilon}I_2\right).
    \end{align*}
\end{lemma}
\begin{proof}
    It can be obtained that
    \begin{align*}       \int_{H>H^*+\epsilon}\pi_1dx &=\frac{1}{Z_1}\int_{H>H^*+\epsilon}e^{-\beta H}dx\\
        &= \frac{1}{Z_1 e^{H^*}}\int_{H>H^*+\epsilon}e^{-\beta (H-H^*)}dx\\
        &\leq \frac{1}{Z_1 e^{H^*}}e^{-(\beta-1) \epsilon}\int_{H>H^*+\epsilon}e^{-(H-H^*)}dx\\
        &\leq \left(\frac{\beta L}{2\pi}\right)^{d/2}e^{-(\beta-1)\epsilon}\int_{H>H^*+\epsilon}e^{-(H-H^*)}dx\\
        &\leq \left(\frac{\beta L}{2\pi}\right)^{d/2}e^{-(\beta-1)\epsilon}I_2,
    \end{align*}
where the last inequality follows from $Z_1\geq e^{-\beta H^*} \left(\frac{2\pi}{\beta L}\right)^{d/2}$ as proved earlier in equation \eqref{lower bound of Z1}. In addition, by Lemma \ref{total variation} and \ref{upper bound of probability}, we can directly compute that 
\begin{align*}
    \int_{H>H^*+\epsilon}\pi_2dx
    &\leq \int_{H>H^*+\epsilon}\pi_1dx+2D_{TV}(\pi_1,\pi_2)\\
    &\leq \left(\frac{\beta L}{2\pi}\right)^{d/2}e^{-(\beta-1)\epsilon}I_2
    +2\left(\frac{\beta L}{2\pi}\right)^{d/2}\left(\mu(c\leq H\leq c+\delta)e^{-\beta (c-H^*)}+e^{-\beta (c-H^*)}I_1 \right)\\
    &=\left(\frac{\beta L}{2\pi}\right)^{d/2}\left(2\mu(c\leq H\leq c+\delta)e^{-\beta (c-H^*)}+2e^{-\beta (c-H^*)}I_1+e^{-(\beta-1)\epsilon}I_2\right),
\end{align*}
which is the desired result.
\end{proof}

In many general cases the condition \eqref{I2} can be verified. For instance, when $H$ follows quadratic growth, the distribution $\pi_1^0\propto e^{-H}$ satisfies a sub-Gaussian tail growth which directly echos with \eqref{I2}.

With Lemma \ref{upper bound of probability} and \ref{upper bound of pi1 and pi2}, we can compute and bound the probability of the LMC algorithms lying at the set $\{x:H(x) \geq H^* + \epsilon\}$ at the $k^{th}$ iteration. To quantitatively demonstrate the improvement brought by landscape modification, we shall compare the upper bound of these large deviation probabilities between the original and the modified iteration simultaneously. 

\begin{proof}[Proof of Theorem \ref{comparison of probability}]
    As an immediate corollary of Lemma \ref{upper bound of probability}, we have respectively
    \begin{align*}
        \mathbb{P}(H(X_k)>H^*+\epsilon)\leq &\int_{H>H^*+\epsilon}\pi_1dx+2D_{TV}(\rho_k,\pi_1),\\
        \mathbb{P}(H(Y_k)>H^*+\epsilon)\leq &\int_{H>H^*+\epsilon}\pi_2dx+2D_{TV}(\rho_k^f,\pi_2).
    \end{align*}
    Collecting Lemma \ref{upper bound of pi1 and pi2}, Proposition \ref{mixing KL}, equation \eqref{TV and KL} and \eqref{L'}, we obtain that
    \begin{align*}
        \mathbb{P}(H(X_k)>H^*+\epsilon) &\leq \int_{H>H^*+\epsilon}\pi_1dx+2D_{TV}(\rho_k,\pi_1)\\
        &\leq  2e^{-\frac{\eta k}{2\beta C_{LSI}}}\sqrt{2D_{KL}(\rho_0\|\pi_{\beta}^0)}+8\sqrt{\eta d L^2\cdot C_{LSI}/\beta}\\
        &\quad + \left(\frac{\beta L}{2\pi}\right)^{d/2}e^{-(\beta-1)\epsilon}I_2,
    \end{align*}
    and
    \begin{align*}
        \mathbb{P}(H(Y_k)>H^*+\epsilon)
        &\leq \int_{H>H^*+\epsilon}\pi_2dx+2D_{TV}(\rho_k^f,\pi_2)\\
        &\leq 2e^{-\frac{\eta k}{2\beta C_{LSI}^f}}\sqrt{2D_{KL}(\rho_0^f\|\pi_{\beta,c,1}^f)}+8\sqrt{\eta d L^{f^2}\cdot C_{LSI}^f/\beta}\\
        &\quad + \left(\frac{\beta L}{2\pi}\right)^{d/2}\left(2\mu(c\leq H\leq c+\delta)e^{-\beta (c-H^*)}+2e^{-\beta (c-H^*)}I_1+e^{-(\beta-1)\epsilon}I_2\right).
    \end{align*}
\end{proof}

\subsection{Proof for reduced excess risk to the global minimum}\label{excess risk to global minimum}
In this section, we will first investigate the excess risk of the modified iteration $Y = \{Y_k\}$ under both the modified objective function $H_{\beta,c,1}^f$ and the original objective function $H$. For the excess risk under the modified objective function, that is, $\mathbb{E}_{Y_k\sim \rho_k^f}[H_{\beta,c,1}^f(Y_k)]-H^*$, it can be decomposed into two parts with the intuition borrowed from \parencite{raginsky2017non}: discretization error and approximation error from sampling, i.e.

\begin{equation}
\mathbb{E}_{Y_k\sim \rho_k^f}[H_{\beta,c,1}^f(Y_k)]-H^*=\underbrace{\mathbb{E}_{Y_k\sim\rho_k^f}[H_{\beta,c,1}^f(Y_k)]-\mathbb{E}_{Y\sim\pi_2}[H_{\beta,c,1}^f(Y)]}_{\text{discretization
error}}+\underbrace{\mathbb{E}_{Y\sim\pi_2}[H_{\beta,c,1}^f(Y)]-H^*}_{\text{sampling error}}.\label{excess risk under G}
\end{equation}
The excess risk of modified iteration under the original objective function, that is, $\mathbb{E}_{Y_k\sim \rho_k^f}[H(Y_k)]-H^*$, allows for the comparison of excess risks under the same criteria between original LMC $X = \{X_k\}$ and the modified LMC $Y = \{Y_k\}$. Specifically, we will study
\begin{equation}
    \mathbb{E}_{Y_k\sim \rho_k^f}[H(Y_k)]-H^*=\underbrace{\mathbb{E}_{Y_k\sim\rho_k^f}[H(Y_k)]-\mathbb{E}_{X\sim\pi_1}[H(X)]}_{\text{modified discretization error}}+\underbrace{\mathbb{E}_{X\sim\pi_1}[H(X)]-H^*}_{\text{sampling error}}.\label{excess risk under H}
\end{equation}
Regarding the discretization error in equation \eqref{excess risk under G}, we can bound it with the following lemma borrowed from Theorem C.1 in \parencite{kinoshita2022improved}. 

\begin{lemma}\label{expectation into W2}
    Suppose $H$ is $L$-$smooth$, then for arbitrary distribution $\nu_1$ and $\nu_2$, random variable $X\sim \nu_1$ and $Y\sim \nu_2$, we have
\begin{equation*}
    \mathbb{E}_{X\sim\nu_1}\left[H\left(X\right)\right]-\mathbb{E}_{Y\sim\nu_2}[H(Y)] \leq L W_2^2\left(\nu_1, \nu_2\right)+\mathbb{E}_{Y\sim\nu_2}[H(Y)]-H^*.
\end{equation*}
\end{lemma}
\begin{proof}
    Let $X\sim \nu_1$ and $Y\sim \nu_2$ be an optimal coupling such that $W_2^2(\nu_1,\nu_2)=\mathbb{E}[\|X-Y\|^2]$, and $\lambda\in \Pi(\nu_1,\nu_2)$ is the optimal distribution. Here $\Pi(\nu_1,\nu_2)$ is the set of probability measures $\xi$ on $\mathcal{B}\left(\mathbb{R}^d\right) \otimes \mathcal{B}\left(\mathbb{R}^d\right)$ which satisfies $\xi\left(A,\mathbb{R}^d\right)=\nu_1(A)$ and $\xi\left(\mathbb{R}^d,B\right)=\nu_2(B)$ for $\forall A,B\in \mathcal{B}\left(\mathbb{R}^d\right)$, where we denote $\mathcal{B}\left(\mathbb{R}^d\right)$ to be the Borel sigma algebra on $\mathbb{R}^d$.
    By Taylor expansion at $Y$, we have
\begin{align*}
    H(X)-H(Y)&\leq \|\nabla H(Y)\|\|X-Y\|+\frac{L}{2}\|X-Y\|^2\\
    &\leq \frac{L}{2}\|X-Y\|^2+\frac{1}{2L}\|\nabla H(Y)\|^2+\frac{L}{2}\|X-Y\|^2\\
    &= L\|X-Y\|^2+\frac{1}{2L}\|\nabla H(Y)\|^2\\
    &\leq L\|X-Y\|^2+H(Y)-H^*,
\end{align*}
where the last inequality utilizes Lemma \ref{H and gradient H}.
Taking expectations with respect to $\lambda$ at both sides, we have
    \begin{equation*}
        \mathbb{E}_{X\sim \nu_1}[H(X)]-\mathbb{E}_{Y\sim \nu_2}[H(Y)]\leq LW_2^2(\nu_1,\nu_2)+\mathbb{E}_{Y\sim \nu_2}[H(Y)]-H^*.
    \end{equation*}
\end{proof}

Note that when $\beta>\beta_0=\max\left\{\frac{1}{\delta^2}, \frac{3}{20\delta},\frac{1}{\delta}\right\}$, we recall from Theorem \ref{theorem 1} that the modified function $H_{\beta,c,1}^f$ is $L^f$-smooth, then we have a direct corollary of Lemma \ref{expectation into W2}.

\begin{corollary}\label{decomposition of H'}
    Under Assumption \ref{assumption 1} to \ref{assumption 5} and \ref{assumption integrable}, for $\beta>\beta_0=\max\left\{\frac{1}{\delta^2}, \frac{3}{20\delta},\frac{1}{\delta}\right\}$, we have 
\begin{equation*}
    \mathbb{E}_{Y_k\sim \rho_k^f}[H_{\beta,c,1}^f(Y_k)]-H^*\leq
L^fW_2^2\left(\rho_k^f,\pi_{\beta,c,1}^f\right)+2\left(\mathbb{E}_{Y\sim \pi_2}[H_{\beta,c,1}^f(Y)]-H^*\right),
\end{equation*}
where $L^f\leq \max\left\{(2(c-H^*)+3\delta)L, \left(\frac{80}{3}(c-H^*)+\frac{92}{3}\delta\right) L, \frac{c-H^*}{2}+\frac{\delta}{8}+\frac{\delta}{2}L\right\}$.
\end{corollary}
\begin{proof} With equation \eqref{excess risk under G}, we can write
\begin{align*}
\mathbb{E}_{Y_k\sim\rho_k^f}[H_{\beta,c,1}^f(Y_k)]-H^* &= \mathbb{E}_{Y_k\sim \rho_k^f}[H_{\beta,c,1}^f(Y_k)]-\mathbb{E}_{Y\sim\pi_2}[H_{\beta,c,1}^f(Y)]+\mathbb{E}_{Y\sim\pi_2}[H_{\beta,c,1}^f(Y)]-H^*\\
&\leq  L^fW_2^2\left(\rho_k^f,\pi_{\beta,c,1}^f\right)+2\left(\mathbb{E}_{Y\sim \pi_2}[H_{\beta,c,1}^f(Y)]-H^*\right),
\end{align*}
where the inequality can be obtained by plugging $\nu_1=\rho_k^f$ and $\nu_2=\pi_2=\pi_{\beta,c,1}^f$ into Lemma \ref{expectation into W2}.
\end{proof}

Regarding the expectation of $H_{\beta,c,1}^f$ with respect to $\pi_{\beta,c,1}^f$, we have the following lemma.

\begin{lemma}\label{modified expectation}
   Under Assumption \ref{assumption 1} to \ref{assumption 5} and \ref{assumption integrable}, for given $0<r<\frac{3\gamma}{2L_1}$ and there is no local minimum other than $x^*$ in $B_r(x^*)$, suppose $\max\limits_{B_r(x^*)}H(x)-H^*=\delta_r\leq c-H^*$, and $\min\limits_{B_r' (x^*)}H(x)-H^*=\Delta_r>0$. Then for step size $0<\eta\leq \frac{\beta C_{LSI}^f}{16L^{f^2}}$,
   $\beta>\beta_0=\max\left\{\frac{1}{\delta},\frac{1}{\delta^2}, \frac{3}{20\delta}\right\}$,
    if there exists $I_3=I_3(c,\delta)<\infty$ such that
    \begin{equation*}
    \int_{H>c+\delta}\left(c-H^*+\delta+\delta \ln \frac{H-c}{\delta}\right)\cdot \frac{2\delta}{H-c}dx<I_3,
    \end{equation*}
    then we have
    \begin{align*}
    \mathbb{E}_{Y\sim \pi_2}[H_{\beta,c,1}^f(Y)]-H^* &\leq \frac{3}{2}\frac{d\omega_d \Gamma(\frac{d}{2}+1)}{\sqrt{\det K}}\left(\frac{2L^f}{\pi}\right)^{d/2}\cdot\frac{1}{\beta}\\
    &\quad +
    \left(\frac{\beta L^f}{2\pi}\right)^{d/2}
    \bigg((c-H^*+\delta)\mu(H\leq c+\delta)e^{-\beta \Delta_r}+e^{-\beta (c-H^*)}\cdot I_3\bigg),
    \end{align*}
    where
    $L^f\leq \max\left\{(2(c-H^*)+3\delta)L, \left(\frac{80}{3}(c-H^*)+\frac{92}{3}\delta\right) L, \frac{c-H^*}{2}+\frac{\delta}{8}+\frac{\delta}{2}L\right\}$.
\end{lemma}
\begin{proof} Utilizing the idea in \parencite{lapinski2019multivariate}, we write
\begin{align}
    \mathbb{E}_{Y\sim \pi_2}[H_{\beta,c,1}^f(Y)]-H^*
    &=\frac{1}{Z_2}\int_{\mathbb{R}^d} Ge^{-\beta G}dx-H^* \nonumber\\
    &=\frac{1}{Z_2}\int_{\mathbb{R}^d}(G-H^*)e^{-\beta G}dx \nonumber\\
    &=\frac{1}{Z_2}\left(\int_{B_r(x^*)}(G-H^*)e^{-\beta G}dx+\int_{B_r'(x^*)}(G-H^*)e^{-\beta G}dx\right), \label{eq:Lem21RHS}
\end{align}
where 
\begin{equation*}
    Z_2=\int_{\mathbb{R}^d} e^{-\beta G}dx\geq e^{-\beta H^*}\int_{\mathbb{R}^d} \exp\left(-\frac{\beta L^f}{2} \|x-x^*\|^2\right)dx=e^{-\beta H^*}\left(\frac{2\pi}{\beta L^f}\right)^{d/2},
\end{equation*}
since $G(x)\leq H^*+\frac{L^f}{2}\|x-x^*\|^2$.

Note that when $\delta_r\leq c$, $B_r(x^*)\subset \{x:H(x)\leq c\}$, and therefore when $x\in B_r(x^*)$ the bounds for the third derivative of $H$ and $G$ remain the same, which means that $L_1$ is a shared parameter in $B_r(x^*)$. As a result, we have
\begin{equation*}
    \left|G-H^*-\frac{1}{2}\|x-x^*\|_K^2\right|\leq \frac{L_1}{6}\|x-x^*\|^3, \quad \forall x\in B_r(x^*).
\end{equation*}
We proceed to handle separately the two terms on the right hand side of \eqref{eq:Lem21RHS}. For the first term, we see that
\begin{align*}
    \int_{B_r(x^*)} &(G-H^*)e^{-\beta G}dx \\
    &\leq 
    e^{-\beta H^*}\int_{B_r(x^*)}\left(\frac{1}{2}\|x-x^*\|_K^2+\frac{L_1}{6}\|x-x^*\|^3\right) \exp \left(-\beta \left(\frac{1}{2}\|x-x^*\|_K^2-\frac{L_1}{6}\|x-x^*\|^3\right)\right)dx\\
    &\leq e^{-\beta H^*}\int_{B_r(x^*)}\left(\frac{3}{4}\|x-x^*\|_K^2\right)
    \exp \left(-\frac{\beta}{4}\|x-x^*\|_K^2\right)dx\\
    &\leq e^{-\beta H^*}\frac{3}{4}\int_{\mathbb{R}^d} \|x-x^*\|_K^2
    \exp \left(-\frac{\beta}{4}\|x-x^*\|_K^2\right)dx\\
    &= e^{-\beta H^*}\frac{3}{4}\frac{d\omega_d}{\sqrt{\det K}}\int_{0}^{\infty}r^{d+1}e^{-\frac{\beta}{4}r^2}dr
    =e^{-\beta H^*}\frac{3}{8}\frac{d\omega_d \Gamma(\frac{d}{2}+1)}{\sqrt{\det K}}\left(\frac{4}{\beta}\right)^{d/2+1},
\end{align*}
where $\omega_d$ is the volume of a $d$-dimensional unit ball in $\mathbb{R}^d$. The second inequality comes from $r<\frac{3\gamma}{2L_1}$.
This leads to 
$$
\frac{1}{Z_2}\int_{B_r(x^*)} (G - H^*)e^{-\beta G}dx\leq 
\frac{3}{2}\frac{d\omega_d \Gamma(\frac{d}{2}+1)}{\sqrt{\det K}}\left(\frac{2L^f}{\pi}\right)^{d/2}\cdot\frac{1}{\beta}=\mathcal{O}(\beta^{-1}).
$$
For the second term on the right hand side of \eqref{eq:Lem21RHS}, the integral on $B_r'(x^*)$ can be split into three parts as follows:
\begin{align*}
    \int_{B_r'(x^*)} &(G-H^*)e^{-\beta G}dx \\
    &= \int_{\{H<c\}\cap B_r'(x^*)} (G-H^*)e^{-\beta G}dx+\int_{c\leq H\leq c+\delta} (G-H^*)e^{-\beta G}dx+\int_{H>c+\delta}(G-H^*)e^{-\beta G}dx\\
    &\leq (c-H^*)\mu(H<c)e^{-\beta (\Delta_r+H^*)}+(c-H^*+\delta)\mu(c\leq H\leq c+\delta)e^{-\beta c}+\int_{H>c+\delta} (G-H^*)e^{-\beta G}dx\\
    &\leq (c-H^*+\delta)\mu(H\leq c+\delta)e^{-\beta (\Delta_r+H^*)}+\int_{H>c+\delta} (G-H^*)e^{-\beta G}dx.
\end{align*}

Note that when $H(x)>c+\delta$, 
since $\beta>\beta_0\geq \frac{1}{\delta}$, by equation \eqref{H and G} we obtain that
\begin{align*}
    c+\frac{1}{\beta}\ln \frac{H-c}{2\delta}\leq G(x)&\leq 
    c+\delta+\frac{1}{\beta}\ln \frac{H-c}{\delta}\\
    &\leq c+\delta+\delta \ln \frac{H-c}{\delta}.
\end{align*}
Plugging these into the integral above, we arrive at 
\begin{align*}
    \int_{H>c+\delta} (G-H^*)e^{-\beta G}dx &\leq \int_{H>c+\delta}\left(c-H^*+\delta+\delta\ln\frac{H-c}{\delta}\right)e^{-\left(\beta c+\ln\frac{H-c}{2\delta}\right)}dx\\
    &=e^{-\beta c}\int_{H>c+\delta}\left(c-H^*+\delta+\delta \ln \frac{H-c}{\delta}\right)\cdot \frac{2\delta}{H-c}dx\\
    &<e^{-\beta c}\cdot I_3,
\end{align*}
and therefore
\begin{equation*}
    \frac{1}{Z_2}\int_{B_r'(x^*)} (G-H^*)e^{-\beta G}dx\leq 
    \left(\frac{\beta L^f}{\pi}\right)^{d/2}
    \left((c-H^*+\delta)\mu(H\leq c+\delta)e^{-\beta \Delta_r}+e^{-\beta (c-H^*)}\cdot I_3\right).
\end{equation*}
Combining with the former part, the desired result follows.
\end{proof}

Note that the Wasserstein distance between the $k^{th}$ iteration of the modified LMC $\rho_k^f$ and its associated modified Gibbs distribution $\pi_{\beta,c,1}^f$ can be estimated through the KL divergence via the Talagrand inequality in Assumption \ref{assumption 6}, then with Proposition \ref{mixing KL} we could obtain the excess risk under $H_{\beta,c,1}^f$.

\begin{theorem}
    Under Assumption \ref{assumption 1} to \ref{assumption 5}, \ref{assumption 6} and \ref{assumption integrable}, for given $0<r<\frac{3\gamma}{2L_1}$ and $x^*$ is the unique local minimum of $H$ in $B_r(x^*)$, suppose $\max\limits_{B_r(x^*)}H(x)-H^*=\delta_r\leq c-H^*$. Assume that $\min\limits_{B_r' (x^*)}H(x)-H^*=\Delta_r>0$, then for step size $0<\eta\leq \frac{\beta C_{LSI}^f}{16L^{f^2}}$, $\beta>\beta_0=\max \left\{\frac{1}{\delta},\frac{1}{\delta^2}, \frac{3}{20\delta}\right\}$, if there exists $I_3=I_3(c,\delta) < \infty$ such that
    \begin{equation*}
    \int_{H>c+\delta}\left(c-H^*+\delta+\delta \ln \frac{H-c}{\delta}\right)\cdot \frac{2\delta}{H-c}dx<I_3,
    \end{equation*}
    then we have
\begin{align*}
    \mathbb{E}_{Y_k\sim\rho_k^f}[H_{\beta,c,1}^f(Y_k)]-H^* &\leq 2L^f\cdot C_{LSI}^f\left(e^{-\frac{\eta k}{\beta C_{LSI}^f}} D_{KL}(\rho_0^f\|\pi_{\beta,c,1}^f)+8\eta d L^{f^2}\cdot C_{LSI}^f/\beta\right)\\
    &\quad +3\frac{d\omega_d \Gamma(\frac{d}{2}+1)}{\sqrt{\det K}}\left(\frac{2L^f}{\pi}\right)^{d/2}\cdot\frac{1}{\beta}\\
    &\quad + 2\left(\frac{\beta L^f}{2\pi}\right)^{d/2}
    \left((c-H^*+\delta)\mu(H\leq c+\delta)e^{-\beta \Delta_r}+e^{-\beta (c-H^*)}\cdot I_3\right),
\end{align*}
where $C_{LSI}^f=\mathcal{O}\left(\beta M+\frac{\beta M^2}{H(m_2)-c} \right)$, and $$L^f \leq \max\left\{(2(c-H^*)+3\delta)L, \left(\frac{80}{3}(c-H^*)+\frac{92}{3}\delta\right) L, \frac{c-H^*}{2}+\frac{\delta}{8}+\frac{\delta}{2}L\right\}.$$
\end{theorem}
\begin{proof}
    First, by Corollary \ref{decomposition of H'}, Proposition \ref{mixing KL} and Lemma \ref{modified expectation}, using the Talagrand inequality in Assumption \ref{assumption 6}, we have 
\begin{align}
    \mathbb{E}_{Y_k\sim\rho_k^f}[H_{\beta,c,1}^f(Y_k)]-H^* &\leq  \frac{2L^f}{T_a^f}\left(e^{-\frac{\eta k}{\beta C_{LSI}^f}} D_{KL}(\rho_0^f\|\pi_{\beta,c,1}^f)+8\eta d L^{f2}\cdot C_{LSI}^f/\beta\right)\label{with Ta} \\
    &\quad +3\frac{d\omega_d \Gamma(\frac{d}{2}+1)}{\sqrt{\det K}}\left(\frac{2L^f}{\pi}\right)^{d/2}\cdot\frac{1}{\beta}\nonumber\\
    &\quad +2\left(\frac{\beta L^f}{2\pi}\right)^{d/2}
    \left((c - H^* +\delta)\mu(H\leq c+\delta)e^{-\beta \Delta_r}+e^{-\beta (c-H^*)}\cdot I_3\right),\nonumber
\end{align}
where $T_a^f$ satisfies that for arbitrary distribution $\rho$, there exists
\begin{equation*}
    \frac{T_a^f}{2} W_2^2(\rho, \pi_{\beta,c,1}^f) \leq D_{KL}(\rho\|\pi_{\beta,c,1}^f).
\end{equation*}

Moreover, by Theorem 1 in \parencite{otto2000generalization}, we note that if a distribution $\nu$ satisfies Log-Sobolev inequality with the constant $C_{LSI}$, then it also satisfies Talagrand inequality with constant $T_a=\frac{1}{C_{LSI}}$. We thus substitute $\frac{1}{C_{LSI}^f}$ to $T_a^f$ in equation \eqref{with Ta} to obtain the desired result.
\end{proof}

We proceed to obtain upper bounds of $\mathbb{E}_{Y_k\sim\rho_k^f}[H(Y_k)]-H^*$ as specified in equation \eqref{excess risk under H}, and compare these results with classic results which give bounds of $\mathbb{E}_{X_k\sim\rho_k}[H(X_k)]-H^*$. In this context, we first quantify both the KL divergence and the Wassterstein distance between the original Gibbs distribution and the modified Gibbs:
\begin{lemma}\label{KL between two distributions}
Under Assumption \ref{assumption 1} to \ref{assumption 5}, \ref{assumption 6} and \ref{assumption integrable}, for given $0<r<\frac{3\gamma}{2L_1}$, and $x^*$ is the unique local minimum of $H$ in $B_r(x^*)$, suppose $\max\limits_{B_r(x^*)}H(x)-H^*=\delta_r\leq c-H^*$. Assume that $\min\limits_{B_r'(x^*)}H(x)-H^*=\Delta_r>0$, 
then we have 
\begin{equation*}
    D_{KL}(\pi_{\beta}^0\|\pi_{\beta,c,1}^f)\leq \left(\frac{\beta L}{2\pi}\right)^{d/2}\left(\mu(c<H\leq c+\delta)e^{-\beta (c-H^*)}+e^{-\beta (c-H^*)}I_1\right),
\end{equation*}
and
\begin{equation*}
    W_2^2(\pi_{\beta}^0,\pi_{\beta,c,1}^f)\leq 
    2C_{LSI}^f\left(\frac{\beta L}{2\pi}\right)^{d/2}\left(\mu(c<H\leq c+\delta)e^{-\beta (c-H^*)}+e^{-\beta (c-H^*)}I_1\right),
\end{equation*}
where
$C_{LSI}^f=\mathcal{O}\left(\beta M+\frac{\beta M^2}{H(m_2)-c} \right)$.
\end{lemma}
\begin{proof}
Recalling that $\pi_1=\pi_{\beta}^0\propto e^{-\beta H}$, $\pi_2=\pi_{\beta,c,1}^f\propto e^{-\beta G}$ and $G\leq H$,  we have 
\begin{align*}
    D_{KL}(\pi_1\|\pi_2)&=\int_{\mathbb{R}^d} \pi_1\ln\left(\frac{\pi_1}{\pi_2}\right)dx\\
    &=\int_{\mathbb{R}^d} \frac{e^{-\beta H}}{Z_1}\ln\left(\frac{e^{-\beta H}/Z_1}{e^{-\beta G}/Z_2}\right)dx\\
    &=\ln \left(\frac{Z_2}{Z_1}\right)-\beta\int_{\mathbb{R}^d} (H-G)\frac{e^{-\beta H}}{Z_1}dx\\
    &\leq \ln\left(\frac{Z_2}{Z_1}\right)\leq \frac{Z_2-Z_1}{Z_1},
\end{align*}
where we have utilized $Z_2/Z_1 \geq 1$ in the last inequality. Now, we observe that both $Z_1$ and $Z_2$ can be decomposed into
\begin{align*}
    Z_1&=\int_{H\leq c}e^{-\beta H}dx+\int_{H>c}e^{-\beta H}dx,\\
    Z_2&=\int_{H\leq c}e^{-\beta G}dx+\int_{c<H\leq c+\delta}e^{-\beta G}dx+\int_{H>c+\delta}e^{-\beta G}dx.
\end{align*}
For the first term of the equations above, as $r<\frac{3\gamma}{2L_1}$,
\begin{align*}
    \int_{H\leq c}e^{-\beta G}dx=\int_{H\leq c}e^{-\beta H}dx
    &=\int_{B_r(x^*)} e^{-\beta H}dx+\int_{B_r'(x^*)\cap\{H\leq c\}}e^{-\beta H}dx\\
    &\leq e^{-\beta H^*} \left(\int_{B_r(x^*)}\exp\left(-\frac{\beta}{4}\|x-x^*\|_K^2\right)dx+\mu(H\leq c)e^{-\beta \Delta_r}\right)\\
    &\leq e^{-\beta H^*} \left(\int_{\mathbb{R}^d}\exp\left(-\frac{\beta}{4}\|x-x^*\|_K^2\right)dx+\mu(H\leq c)e^{-\beta \Delta_r}\right)\\
    &= e^{-\beta H^*} \left(\frac{d\omega_d}{\sqrt{\det K}}\int_{0}^{\infty}r^{d-1}e^{-\frac{\beta}{4}r^2}dr+\mu(H\leq c)e^{-\beta \Delta_r}\right)\\
    &=e^{-\beta H^*} \left(\frac{1}{2}\frac{d\omega_d \Gamma(\frac{d}{2})}{\sqrt{\det K}}\left(\frac{4}{\beta}\right)^{d/2}+
    \mu(H\leq c)e^{-\beta \Delta_r}\right),
\end{align*}
 which gives the upper bound of $\int_{H\leq c}e^{-\beta H}\,dx$.
 As for the lower bound of $Z_1$, by equation \eqref{lower bound of Z1} we have
\begin{equation*}
    Z_1=\int_{\mathbb{R}^d} e^{-\beta H}dx\geq e^{-\beta H^*}\left(\frac{2\pi}{\beta L}\right)^{d/2}.
\end{equation*}
Moreover, we note that 
\begin{align*}
    \int_{c<H\leq c+\delta}e^{-\beta G}dx+\int_{H>c+\delta}e^{-\beta G}dx
    &\leq \mu(c<H\leq c+\delta)e^{-\beta c}+e^{-\beta c}\int_{H>c+\delta}\frac{2\delta}{H-c}dx\\
    &\leq \mu(c<H\leq c+\delta)e^{-\beta c}+e^{-\beta c}I_1.
\end{align*}
Collecting the above results yields
\begin{align*}
    D_{KL}(\pi_1\|\pi_2)\leq \frac{Z_2-Z_1}{Z_1}&\leq 
    \frac{\int_{c<H\leq c+\delta}e^{-\beta G}dx+\int_{H>c+\delta}e^{-\beta G}dx}{Z_1}\\
    &\leq \left(\frac{\beta L}{2\pi}\right)^{d/2}\left(\mu(c<H\leq c+\delta)e^{-\beta (c-H^*)}+e^{-\beta (c-H^*)}I_1\right)\\
    &=\mathcal{O}\left(\beta^{d/2}e^{-\beta (c-H^*)}\right).
\end{align*}

Using the Talagrand inequality again, and note that $T_a^f$ can be substituted with $\frac{1}{C_{LSI}^f}$, the case for the Wasserstein distance can be obtained analogously.
\end{proof}

With Lemma \ref{H and gradient H} and \ref{KL between two distributions}, we bound the expectation of $H(Y_k)$.

\begin{proof}[Proof of Theorem \ref{result of expectation under H}]
    Plugging $\nu_1=\rho_k^f$ and $\nu_2=\pi_{\beta}^0$ into Lemma \ref{expectation into W2}, we have
\begin{align}
    \mathbb{E}_{Y_k\sim\rho_k^f}[H(Y_k)]-H^*&= \mathbb{E}_{Y_k\sim\rho_k^f}[H(Y_k)]-\mathbb{E}_{X\sim\pi_1}[H(X)]+\mathbb{E}_{X\sim\pi_1}[H(X)]-H^*\nonumber\\
    &\leq LW_2^2(\rho_k^f,\pi_{\beta}^0)+2\left(\mathbb{E}_{X\sim\pi_1}[H(X)]-H^*\right).\label{decomposition of original expectation}
\end{align}

By Proposition \ref{mixing KL} and Lemma \ref{KL between two distributions}, using the Talagrand inequality with respect to $\pi_{\beta,c,1}^f$, we have
\begin{align*}
    W_2^2(\rho_k^f,\pi_{\beta}^0) &\leq 2\left(W_2^2(\rho_k^f,\pi_{\beta,c,1}^f)^2+W_2^2(\pi_{\beta}^0,\pi_{\beta,c,1}^f)^2\right) \\
    &\leq 4C_{LSI}^f\left(e^{-\frac{\eta k}{\beta C_{LSI}^f}} D_{KL}(\rho_0^f\|\pi_{\beta,c,1}^f)+8\eta d L^{f2}\cdot C_{LSI}^f/\beta\right)\\
    &\quad +4C_{LSI}^f\left(\frac{\beta L}{2\pi}\right)^{d/2}\left(\mu(c<H\leq c+\delta)e^{-\beta (c-H^*)}+e^{-\beta (c-H^*)}I_1\right).
\end{align*}

As for $\mathbb{E}_{X\sim\pi_1}[H(X)]$, the computation is similar to that as presented in the proof of Lemma \ref{modified expectation}. Precisely, it can be calculated that 
\begin{align*}
    \mathbb{E}_{X\sim\pi_1}[H(X)]-H^*
    &= \frac{1}{Z_1}\left(\int_{H\leq c} (H-H^*) e^{-\beta H}dx+\int_{H>c} (H-H^*)e^{-\beta H}dx\right)\\
    &\leq \frac{3}{2}\frac{d\omega_d \Gamma(\frac{d}{2}+1)}{\sqrt{\det K}}\left(\frac{2L}{\pi}\right)^{d/2}\cdot\frac{1}{\beta}\\
    &\quad +\left(\frac{\beta L}{2\pi}\right)^{d/2}\left((c-H^*)\mu(H\leq c)e^{-\beta \Delta_r}
    +\int_{H>c} (H-H^*)e^{-\beta (H-H^*)}dx\right)\\
    &\leq \frac{3}{2}\frac{d\omega_d \Gamma(\frac{d}{2}+1)}{\sqrt{\det K}}\left(\frac{2L}{\pi}\right)^{d/2}\cdot\frac{1}{\beta}
    \\
    &\quad +\left(\frac{\beta L}{2\pi}\right)^{d/2}\left((c-H^*)\mu(H\leq c)e^{-\beta \Delta_r}+e^{-(\beta-1)(c-H^*)}I_4\right).
\end{align*}

The desired result follows via plugging the above into equation \eqref{decomposition of original expectation}.
\end{proof}

\section{Table of commonly used notations}\label{appendix A}

In Table \ref{tab:notation}, we summarize important notations that are frequently used in this paper.
\begin{table}[ht]
    \centering
    \caption{Table of commonly used notations}
    \begin{tabular}{c|c|c}
    \hline
        Notation & Description & Definition \\
    \hline
        $G(x)$ & modified objective function & $G(x)=H_{\beta,c,1}^f(x)$\\
        $X_k$ & $k^{th}$ iteration of original LMC & \\
        $Y_k$ & $k^{th}$ iteration of modified LMC & \\
        $\rho_k$ & distribution of $X_k$ & $X_k\sim \rho_k$\\
        $\rho_k^f$ & distribution of $Y_k$ & $Y_k\sim \rho_k^f$\\
        $\pi_{\beta}^0$ & Gibbs distribution of original dynamics & $\pi_{\beta}^0\propto e^{-\beta H}$\\
        $\pi_{\beta,c,1}^f$ & Gibbs distribution of modified dynamics & 
        $\pi_{\beta,c,1}^f\propto e^{-\beta G}$\\
        $\pi_1$ & shorthand notation of $\pi_{\beta}^0$ & $\pi_1=\pi_{\beta}^0$\\
        $\pi_2$ & shorthand notation of $\pi_{\beta,c,1}^f$ & $\pi_2=\pi_{\beta,c,1}^f$\\
        $Z_1$ & normalization constant of $\pi_1$ & $Z_1=\int e^{-\beta H}dx$\\
        $Z_2$ & normalization constant of $\pi_2$ & $Z_2=\int e^{-\beta G}dx$\\
        $C_{PI}$ & Poincar\'e constant of $\pi_1$ & \\
        $C_{LSI}$ & Log-Sobolev constant of $\pi_1$ & \\
        $C_{PI}^f$ & Poincar\'e constant of $\pi_2$ & \\
        $C_{LSI}^f$ & Log-Sobolev constant of $\pi_2$ & \\
        $L$ & $H$ is $L$-smooth & $\|\nabla H(x)-\nabla H(y)\|\leq L\|x-y\|$ \\
        $L^f$ & $G$ is $L^f$-smooth & $\|\nabla G(x)-\nabla G(y)\|\leq L^f\|x-y\|$\\
        $L_1$ & $\nabla^2 H$ is $L_1$-Lipschitz & $\|\nabla^2 H(x)-\nabla^2 H(y)\|\leq L_1\|x-y\|$ \\
        $T_a$ & Talagrand constant of $\pi_1$ & \\
        $T_a^f$ & Talagrand constant of $\pi_2$ & \\
        $\mathcal{S}=\{m_k\}_{k=1}^S$ & local minima of $H$ & $H(m_1)<H(m_2)\leq H(m_i), 3\leq i\leq S$\\
        $\mathcal{S}^f=\{m_k^f\}_{k=1}^{S^f}$ & local minima of $G$ & It is proved that $m_k^f=m_k$, $\forall k$\\
        $s_{1,2}$ & saddle point between $m_1$ and $m_2$ & \\
        $s_{1,2}^f$ & saddle point between $m_1^f$ and $m_2^f$ & It is proved that $s_{1,2}^f=s_{1,2}$ and $S=S^f$ \\
        $x^*$ & the unique global minimum of $H$ (and $G$) & $x^*=m_1$\\
        $H^*$ & global minimum value of $H$ & $H^*=H(x^*)$\\
        $M$ & energy barrier of $H$ & $M=H(s_{1,2})-H(m_2)$\\
        $M^f$ & energy barrier of $G$ & $M^f=G(s_{1,2}')-G(m_2')$\\
        $K$ & Hessian of $H$ at $m_1$ & $K=\nabla^2 H(m_1)$\\
        $\gamma$ & the smallest eigenvalue of $\nabla^2 H(m_1)$ & $\gamma=\lambda_{min}(K)$\\
        $\|x\|_A$ & matrix norm of $x$ with respect to matrix $A$ & $\|x\|_A^2=\frac{1}{2}x^TAx$\\
        $\|B\|$ & operator norm of matrix $B$ & \\
        $\omega_d$ & volume of the unit ball in $\mathbb{R}^d$ & \\
        $\mu$ & Lebesgue measure in $\mathbb{R}^d$ & \\
        $\Pi(\rho,\nu)$ & the set of possible couplings between $\rho$ and $\nu$ & \\
        $D_{KL}(\rho\|\nu)$ & KL divergence of $\rho$ with respect to $\nu$ & $D_{KL}(\rho\|\nu)=\int \rho \ln \frac{\rho}{\nu}dx$\\
        $D_{TV}(\rho,\nu)$ & Total variation distance between $\rho$ and $\nu$ & $D_{TV}(\rho,\nu)=\frac{1}{2}\int |\rho-\nu|dx$ \\
        $W_2(\rho,\nu)$ & Wasserstein distance between $\rho$ and $\nu$ & $W_2^2(\rho,\nu)=\inf\limits_{\xi\in\Pi(\rho,\nu)}\int\|x-y\|^2\xi(dx,dy)$
        \\
    \hline
    \end{tabular}
    \label{tab:notation}
\end{table}
\end{document}